\documentclass[12pt]{amsart}

\usepackage{amsmath}
\usepackage{amssymb}
\usepackage{latexsym}
\usepackage{tikz}
\usepackage{enumerate}
\input xypic
\xyoption {all}

\usepackage{color}

\theoremstyle{plain}
\newtheorem{thm}{Theorem}[section]

\newtheorem{cor}[thm]{Corollary}

\newtheorem{remark}[thm]{Remark}

\theoremstyle{definition}
\newtheorem*{ack}{Acknowledgments}
\newtheorem{defn}[thm]{Definition}
\newtheorem{rem}[thm]{Remark}

\numberwithin{equation}{section}

\newcommand{\al}{\alpha}
\newcommand{\be}{\beta}

\newcommand{\de}{\delta}
\newcommand{\e}{\varepsilon}

\newcommand{\la}{\lambda}
\newcommand{\m}{\mu}

\newcommand{\om}{\omega}

\newcommand{\Om}{\Omega}
\newcommand{\La}{\Lambda}

\newcommand{\R}{\mathbb{R}}
\newcommand{\N}{\mathbb{N}}
\newcommand{\ACq}{AC}
\newcommand{\AC}{AC^{(n)}(\Omega,\mathbb{R}^l)}
\newcommand{\ACH}{AC_H^{(n)}(\Omega,\mathbb{R}^l)}
\newcommand{\LL}{\mathcal{L}}
\newcommand{\po}{\text{-}}
\newcommand{\lip}{\mbox{Lip}}
\newcommand{\dist}{\mbox{dist}}

\newcommand{\inter}{\mbox{Int}}
\newcommand{\diam}{\mbox{diam}}

\newcommand{\leb}{\mathcal{L}}

\newcommand{\linspan}{\operatorname{lin\ span}}
\providecommand{\osc}{\mathop{\rm osc}\nolimits}
\providecommand{\dist}{\mathop{\rm dist}\nolimits}
\providecommand{\supp}{\mathop{\rm supp}\nolimits}

\providecommand{\diam}{\mathop{\rm diam}\nolimits}

\providecommand{\lip}{\mathop{\rm lip}\nolimits}
\providecommand{\sgn}{\mathop{\rm sgn}\nolimits}

\newcommand{\lb}{\label}

\newcommand{\lra}{\longrightarrow}

\newcommand{\buo}{without loss of generality }
\newcommand{\buoo}{without loss of generality}
\newcommand{\Buo}{Without loss of generality}

\newcommand{\DEF}{\buildrel {\mbox{\tiny def}}\over =}

\begin{document}

\title[Interval based generalizations of absolute continuity]{On interval based generalizations of absolute continuity for functions on $\R^n$}



\author{Michael Dymond}

\address{School of Mathematics\\
University of Birmingham\\
Birmingham, B15 2TT, UK}

\email{\texttt{dymondm@maths.bham.ac.uk}}

\author{Beata~Randrianantoanina}

\address{Department of Mathematics\\
Miami University\\
Oxford, OH 45056, USA}

\email{\texttt{randrib@miamioh.edu}}

\author{Huaqiang~Xu}

\address{Department of Mathematics\\
Miami University\\
Oxford, OH 45056, USA}

\email{\texttt{xuhuaqiang1990@gmail.com}}




\begin{abstract}
We study  notions of   absolute continuity 
for functions defined on $\mathbb{R}^n$similar to the notion of  $\alpha$-absolute continuity in the  sense of Bongiorno. 
We confirm a conjecture  of Mal\'y that 1-absolutely continuous functions
do not need to be differentiable a.e., and we show several other  pathological examples of functions in this class.  We establish containment relations of the class $1\po AC_{\rm WDN}$ which  consits of all functions in $1\po AC$ which are in the Sobolev space $W^{1,2}_{loc}$,  are differentiable a.e. and satisfy the Luzin (N) property, with previously studied classes of absolutely continuous functions.
\end{abstract}

\maketitle

\section{Introduction}
\label{intro}

The classical Vitali's definition says that when $\Om\subseteq \R$, a function $f:\Om\lra \R$  is {\it absolutely continuous}  if for all $\e>0$, there exists $\de>0$ so that for every finite collection of disjoint intervals $\{[a_i,b_i]\}_{i=1}^k \subset \Om$ we (below $\LL^n$  denotes the Lebesgue measure on $\R^n$)
\begin{equation}\lb{hypp}
\sum_{i=1}^k \LL^1[a_i,b_i]<\de  \Rightarrow \sum_{i=1}^{k} |f(a_i)-f(b_i)|<\e.
\end{equation}

The study of the space of absolutely continuous functions on $[0,1]$ and their generalizations to domains in $\R^n$ is connected to the problem of finding regular subclasses of Sobolev spaces which goes back to Cesari and Calder{\'o}n \cite{C41,Cal51}. On the other hand, the Banach space of generalized absolutely continuous functions on $[0,1]$ is closely related to the famous James space, and it is an example of a separable space   not containing $\ell_1$ but with a non-separable dual \cite{LS75,K84}. Moreover it has a very rich subspace structure \cite{AAK11,AMP03}, but several questions about the Banach space structure of this space remain open, see  \cite{AAK11}.

There are several natural ways of generalizing the definition of  absolute continuity for functions of several variables (cf. \cite{CA1933,Hobson,RR,Ziemer,AAK11}).

One approach is to replace  the intervals in the antecedent  of \eqref{hypp} by balls in $\R^n$ and differences in the conclusion of \eqref{hypp} by oscillations of $f$ on the images of balls from \eqref{hypp}. This approach goes back to Banach, Vitali and Tonelli \cite{Banach1925,Vitali1926,Tonelli1926} (cf.     \cite{Hobson}).
More recently  Mal\'y \cite{Maly1999} suggested another fruitful approach which is to replace 
the intervals in the antecedent  of \eqref{hypp} by balls of a selected norm in $\R^n$ and replace sums in the conclusion of \eqref{hypp}  by sums of oscillations raised to the power equal to the dimension of the domain space. This generalized notion gives functions in the Sobolev space $W^{1,n}_{loc}(\Om)$, when $\Om\subset \R^n$, and it
 has been extensively studied by Mal\'y, Cs\"ornyei, Hencl and Bongiorno \cite{Csornyei2000,Hencl2002,Hencl2003,HenclandMaly2003,Bongiorno2009a}.  Cs\"ornyei \cite{Csornyei2000} proved that this notion does depend on the shape of the balls substituted for intervals in  \eqref{hypp}. Thus the incomparable classes $\mathcal{Q}\po\ACq$ and $\mathcal{B}\po\ACq$ are defined where cubes (i.e. balls in the $\ell^n_{\infty}$-norm) or Euclidean balls are used, respectively. Hencl \cite{Hencl2002} introduced  a shape-independent class $\ACq_{H}$ (see Definition~\ref{qac}) 
which contains both  classes $\mathcal{Q}\po\ACq$ and $\mathcal{B}\po\ACq$ and so that $\ACq_{H}$ is contained in the Sobolev space  $W^{1,n}_{loc}$  and that all functions in $\ACq_{H}$ are differentiable a.e. and satisfy the Luzin (N) property and the change of variable formula.

Bongiorno \cite{Bongiorno2005} introduced another  generalization of Vitali's classical definition  for functions of several variables, which is simultaneously similar to Arzel\`a's notion of bounded variation for functions on $\R^2$, cf. \cite{CA1933}, and to Mal\'y's definition \cite{Maly1999}.

\begin{defn} 
\label{alphaac}(Bongiorno \cite{Bongiorno2005})
Let $0<\al<1$. A function $f:\Om\to\R^l$, where $\Om\subset\R^n$ is open, is said to be
\textit{$\al$-absolutely continuous} (denoted $\al\po\AC$ or   $\al\po\ACq$) if for all $ \varepsilon>0$, 
 there exists $\de>0$, such that for any finite collection of disjoint $\al$-regular intervals 
$\{ [\mathbf{a}_i,\mathbf{b}_i]\subset\Om \}_{i=1}^k$ 
we have 
\begin{equation}\lb{alac}
\sum_{i=1}^k \LL^n ([\mathbf{a}_i,\mathbf{b}_i])<\de \Rightarrow\sum_{i=1}^k|f(\mathbf{a}_i)-f(\mathbf{b}_i)|^n<\e.
\end{equation}
\end{defn}
Here, for $\mathbf{a}\in\R^l$, $|\mathbf{a}|$ denotes the Euclidean norm of $\mathbf{a}$, and we say that  an interval $[\mathbf{a},\mathbf{b}]\DEF \{\mathbf{x}=(x_\nu)_{\nu=1}^n\in\R^n\text{: }a_\nu\leq x_\nu\leq b_\nu,\nu=1,\dots,n\}$ is {\it $\al$-regular} if
$$
\frac{\LL^n([\mathbf{a},\mathbf{b}])}{(\max_\nu|a_\nu-b_\nu|)^n}\geq \al.
$$
Bongiorno \cite{Bongiorno2005} showed that for all $0<\al<1$,
$$
Q\po\AC\subsetneq \al\po\AC \subsetneq \ACH.
$$

In 2012, Mal\'y \cite{Maly} asked us about the properties of absolutely continuous functions in a sense similar to Definition~\ref{alphaac}, but without restriction to $\al$-regular intervals for a specified $0<\al<1$. 
This question led us to the following definitions:
\begin{defn} \label{oldzeroac}
We say that a function $f:\Om \rightarrow \R^l$ ($\Omega \subset \R^n$ open) is {\it $0$-absolutely continuous}, denoted $0\po\AC$ or  $0\po\ACq$, (resp.\! {\it strongly $0$-absolutely continuous}, denoted strong-$0\po\AC$ or  strong-$0\po\ACq$) if for every $ \e > 0$, there exists $\delta>0$, such that for any finite collection of disjoint arbitrary intervals $\{ [\mathbf{a}_i,\mathbf{b}_i] \subset\Omega\}_{i=1}^k  $  we have 
\begin{equation}\lb{0AC}
\sum_{i=1}^k (\max_\nu|\mathbf{a}_{i,\nu}-\mathbf{b}_{i,\nu}|)^n <\delta\Rightarrow\sum_{i=1}^k|f(\mathbf{a}_i)-f(\mathbf{b}_i)|^n<\e,
\end{equation}
respectively,
\begin{equation}\lb{strong0AC}
\sum_{i=1}^k \LL^n ([\mathbf{a}_i,\mathbf{b}_i])<\delta\Rightarrow\sum_{i=1}^k|f(\mathbf{a}_i)-f(\mathbf{b}_i)|^n<\e.
\end{equation}
\end{defn}

Note that the antecedent of implication \eqref{0AC} is equivalent to the antecedent of \eqref{alac}, since intervals in \eqref{alac} are $\al$-regular for a fixed $\al$. The antecedent of \eqref{strong0AC} is much weaker since there is no assumption of $\al$-regularity of intervals.

We show that, when $n\ge 2$, the condition \eqref{strong0AC} characterizes constant functions, and   \eqref{0AC} characterizes Lipschitz functions (Theorem~\ref{0aceqcons}). 

The main goal of this paper is to study an analog of Bongiorno's notion for $\al=1$.

\begin{defn} \label{oneac}
We say that a function $f:\Om \rightarrow \R^l$ ($\Om \subset \R^n$ open) is {\it $1$-absolutely continuous}, denoted $1\po\AC$  or $1\po\ACq$,  if  for every $ \e > 0$, there exists $\delta>0$, such that for any finite collection of disjoint $1$-regular intervals $\{[\mathbf{a}_i,\mathbf{b}_i] \subset\Om\}_{i=1}^k $ we have $$\sum_{i=1}^k \LL^n ([\mathbf{a}_i,\mathbf{b}_i])<\de\Rightarrow \sum_{i=1}^k|f(\mathbf{a}_i)-f(\mathbf{b}_i)|^n<\e.$$
\end{defn}

We show that the class $1\po AC$ is not contained in $AC_H$ or even in the Sobolev space $W^{1,n}_{loc}(\Om)$. We show that, similarly as the Sobolev space  $W^{1,n}(\Om)$, when $\Om\subset\R^n$ and $n>1$, cf. \cite{EG,Maly1995lusin},  the class $1\po AC $ contains functions with  pathological properties such as:

(i) compactly supported, but unbounded,

(ii)  bounded but discontinuous,

(iii)  continuous but nowhere differentiable,

(iv)  differentiable but without the Luzin~(N) property,

(v)  differentiable, with the Luzin~(N) property but not in the class $W^{1,n}_{loc}$.

However we prove that every  function in $1\po AC$ has a directional derivative in the direction $(1,\dots,1)$ at a.e. point of the  domain (Theorem~\ref{diff11}). 

Moreover the class $1\po AC$ is useful for the study of the Bongiorno's classes $\al\po AC$. Namely, in \cite{RXB} it is proved that 
\begin{thm}(\cite[Theorem~3.2]{RXB}) \label{1capH}
For all $0<\al<1$, $$\al\po AC=1\po AC\cap AC_H.$$
\end{thm}

We finish the paper by showing where the class $1\po AC_{\rm WDN}$ which  consits of all functions in $1\po AC$ which are in the Sobolev space $W^{1,n}_{loc}$,  are differentiable a.e. and satisfy the Luzin (N) property,  fits in the hierarchy of previously studied classes. Namely we prove that (Theorem~\ref{thmcontainments}):
\begin{equation}\label{introcont}
 \begin{split}
 &\mathcal{Q}\po\ACq \subsetneq \al\po\ACq= 1\po\ACq_{\rm WDN}\cap \ACq_H\subsetneq 1\po\ACq_{\rm WDN}\\& \subsetneq 1\po\ACq_{\rm WDN}\cup \ACq_H \subsetneq \linspan( 1\po\ACq_{\rm WDN}\cup \ACq_H)\subseteq1\po\ACq_{\rm HWDN},
\end{split}
\end{equation}
where $1\po\ACq_{\rm HWDN}$ denotes the set of functions in $1\po\ACq_H$ (see Defnition~\ref{1AClambda} and Remark~\ref{def1ACH}) which are in the Sobolev space $W^{1,2}_{loc}$,  are differentiable a.e. and satisfy the Luzin (N) property. We pose a few related open questions in Remark~\ref{open}.

 On the other hand we observe that a  small adjustment of the function constructed by  Cs{\"o}rnyei in \cite[Theorem~2]{Csornyei2000}  (see  \eqref{bacnotone} below) shows that
\begin{equation}\label{introbacnotone}
1\po\ACq_{\rm WDN}\setminus \mathcal{B}\po\ACq \ne\emptyset,  \ \ \ \text{ and }\ \ \ \mathcal{B}\po\ACq \setminus 1\po\ACq_{\rm WDN}\ne\emptyset.
\end{equation}

\section{Preliminaries}
\label{prelres}

 Let $C_0(\R^n,\R^l)$ denote the set of all continuous functions $f:\R^n\rightarrow\R^l$ with compact support. For $f\in C_0(\R^n,\R^l)$, and a measurable set $A\subset\R^n$, let $\osc(f,A)$ denote the oscillation of $f$ on $A$, i.e.
\[
\osc(f,A) = \diam f(A).
\]
Let $K_0 \subset \R^n$ be a fixed symmetric closed convex set with non-empty interior, and let $\mathcal{K}$ denote the set of all balls of $\R^n$ in the norm defined by set $K_0$, i.e., 
$$
\mathcal{K}=\{a+rK_0:a\in\R^n, r>0\}.
$$
\begin{defn}\label{defKAC} (Cs{\"o}rnyei \cite{Csornyei2000})
We  say that a function $f\in C_0(\R^n,\R^l)$ is \textit{absolutely continuous with respect to $\mathcal{K}$} (denoted $f\in\mathcal{K}\po\ACq$) if for every $\e>0$, there exists $\de>0$ such that for every finite collection of disjoint sets $\{K_i\}_{i=1}^k \subset \mathcal{K}$,
$$
\sum\limits_{i=1}^{k} \LL^n(K_i)<\de \Rightarrow \sum\limits_{i=1}^{k} \osc^n(f,K_i)<\e.
$$
\end{defn}

 Mal\'y \cite{Maly1999} considered functions absolutely continuous with respect to the family  $\mathcal{B}$ of Euclidean balls in $\R^n$ and showed that all functions in $\mathcal{B}\po\ACq$ are differentiable a.e. and satisfy the change of variables formula, similarly as functions in $\mathcal{Q}\po\ACq$, where $\mathcal{Q}$ denotes the  family  of cubes, i.e. balls in the $\ell^n_{\infty}$-norm.
Cs{\"o}rnyei \cite{Csornyei2000} and Hencl and Mal\'y \cite{HenclandMaly2003} showed that the classes $\mathcal{B}\po\ACq$ and $\mathcal{Q}\po\ACq $ are incomparable.

In 2002, Hencl \cite{Hencl2002} introduced the following shape-independent class of absolutely continuous functions  which contains both classes $ \mathcal{Q}\po\ACq$ and $\mathcal{B}\po\ACq$.

\begin{defn} \label{qac} (Hencl \cite{Hencl2002})
We say that a function $f:\Om\to\R^l$ ($\Om\subset\R^n$ open)  is  in $\ACH$ (briefly $AC_H$) if there exists $\la\in(0,1)$  (equivalently, for all $\la\in(0,1)$) so that for all $\e>0$, there exists $\de>0$ so that for any finite collection of disjoint closed balls $\{B(\mathbf{x}_i,r_i) \subset\Om\}_{i=1}^k $  $$\sum\limits_{i=1}^k \LL^n (B(\mathbf{x}_i,r_i))<\de\Rightarrow\sum\limits_{i=1}^k \osc^n(f,B(\mathbf{x}_i,\la r_i))<\e.$$
\end{defn}

Hencl \cite{Hencl2002} proved that $AC_H\subset W^{1,n}_{loc}$ and that  all functions in $AC_H$ are  differentiable a.e. and satisfy the Luzin (N) property and  the change of variables formula.

\section{Classes $0\po\ACq$ and strong-$0\po AC$}
\label{zeroacprop}

The main result of this section is the following. 

\begin{thm} \label{0aceqcons}
Let $n\ge 2$ and  $f:\Om\to\R^l$ be a function from an open connected set $\Om\subset\R^n$. Then:
\begin{enumerate}
\item[(a)] $f\in strong\po 0\po AC$  if and only if $f$ is constant on $\Om$,
\item[(b)] $f\in 0\po AC$  if and only if $f$ is Lipschitz.
\end{enumerate}
\end{thm}

\begin{proof}[Proof of {\rm (a)}]
We show that every strong-$ 0\po\AC $  function  $f$ is constant. It is clear that if  $f\in$ strong-$ 0\po\AC $ then  $f$ is continuous.  

Let $\mathbf{a}\in\Om$. We claim that the set $\Om_\mathbf{a}\DEF\{\mathbf{b}\in \Om: f(\mathbf{b})=f(\mathbf{a})\}$ is equal to $\Om$. Since $f$ is  continuous, $\Om_\mathbf{a}$ is closed in $\Om$. We will show that $\Om_\mathbf{a}$ is open.

Fix $\e>0$. Let $r>0$ be such that $B(\mathbf{a},r)\subset \Om$, and let $\mathbf{b}\in B(\mathbf{a},r)$.
For each  $j=0,1,\dots,n-1,$ let
\begin{eqnarray*}
(\mathbf{c}^{(j)})_i = 
\begin{cases}
b_i, \text { if } i\leq j, \\
 a_i,  \text { if }  i>j.
\end{cases}
\end{eqnarray*}
Then $\mathbf{c}^{(0)}=\mathbf{a}$, $\mathbf{c}^{(n-1)}=\mathbf{b}$, and  for each $j\le n-1$,   $\mathbf{c}^{(j)}\in  B(\mathbf{a},r) $, and $\mathbf{c}^{(j)}$ and $\mathbf{c}^{(j+1)}$ differ only on the $(j+1)$-th coordinate.
Since $f$ is  continuous,  there exists $\de_1>0$, such that for all $j\le n-1$ and  all $\mathbf{x}$ with $|\mathbf{x}-\mathbf{c}^{(j)}|<\de_1$ we have
\begin{equation}\lb{fcont}
 |f(\mathbf{c}^{(j)})-f(\mathbf{x})|<\frac{\e}{2n}.
\end{equation}

For any $t>0$, let
\begin{align*}
(\mathbf{x}^{(j+1)}(t))_i = 
\begin{cases}
(\mathbf{c}^{(j+1)})_{j+1}, &\text { if } i= j+1, \\
(\mathbf{c}^{(j+1)})_i +t\sgn((\mathbf{c}^{(j+1)})_{j+1}-(\mathbf{c}^{(j)})_{j+1}),  &\text { if }  i\ne j+1.
\end{cases}
\end{align*}

Then 
\begin{equation}\lb{xc}
|\mathbf{x}^{(j+1)}(t)- \mathbf{c}^{(j+1)}|=\sqrt{n-1}\ t,
\end{equation}
and 
the Lebesgue measure of the interval with endpoints $\mathbf{x}^{(j+1)}$ and $\mathbf{c}^{(j)}$ is equal to $t^{n-1} |\mathbf{a}_{j+1}-\mathbf{b}_{j+1}|$.
By \eqref{strong0AC}, there exists $\de_2>0$, such that   if  $t^{n-1} |\mathbf{a}_{j+1}-\mathbf{b}_{j+1}|\le t^{n-1}r<\de_2,$
then
\begin{equation}\lb{fcjx}
 |f(\mathbf{c}^{(j)})-f(\mathbf{x}^{(j+1)}(t))|^n<\left(\frac{\e}{2n}\right)^n.
\end{equation}
Now let $\de=\min(\de_1/\sqrt{n-1},\sqrt[n-1]{\de_2/r})$ and  $0<t<\de$. Then, by \eqref{fcont}, \eqref{xc} and \eqref{fcjx}, we have 
\begin{align*}
|f(\mathbf{c}^{(j)})-f(\mathbf{c}^{(j+1)})|&\leq |f(\mathbf{c}^{(j)})-f(\mathbf{x}^{(j+1)}(t))|+|f(\mathbf{x}^{(j+1)}(t))-f(\mathbf{c}^{(j+1)})|\\
&<\frac{\e}{2n}+\frac{\e}{2n}=\frac{\e}{n}.
\end{align*}
Thus 
\begin{align*}
|f(\mathbf{a})-f(\mathbf{b})|&\leq\sum_{j=0}^{n-1} |f(\mathbf{c}^{(j)})-f(\mathbf{c}^{(j+1)})|<n\frac{\e}{n}= \e.
\end{align*}
Since $\e$ is arbitrary, we get $f(\mathbf{b})=f(\mathbf{a})$, and thus $B(\mathbf{a},r)\subseteq \Om_\mathbf{a}$.
\end{proof}

\begin{proof}[Proof of {\rm (b)}]
Since the Euclidean norm  and the $\ell_1$-norm  $\|\cdot\|_1$ are equivalent on  $\R^n$, we can \buo replace $|\cdot|$ with $\|\cdot\|_1$  in the definition of 0-absolute continuity, which we will do to simplify some computations.

It is not difficult to see that Lipschitz functions are 0-absolutely continuous. Indeed, suppose that $f:\Om\to\R^l$ ($\Om\subset\R^n$ open) is  Lipschitz with constant  $M>0$.   Then for any $\e>0$ and for any finite collection of non-overlapping arbitrary intervals $\{[\mathbf{a}_i,\mathbf{b}_i]\subset \Om\}$  with 
$$
\sum\limits_{i=1}^k (\max_\nu|a_{i,\nu}-b_{i,\nu}|)^n <\frac{\e}{(Mn)^n},
$$
we have 
\begin{align*}
\sum\limits_{i=1}^k \|f(\mathbf{a}_i)-f(\mathbf{b}_i)\|_1^n &<\sum\limits_{i=1}^k M^n\|\mathbf{a}_i-\mathbf{b}_i\|_1^n \\
&\leq \sum\limits_{i=1}^k M^n\cdot(n\max_\nu|a_{i,\nu}-b_{i,\nu}|)^n \\
&< (Mn)^n\cdot \frac{\e}{(Mn)^n} =\e,
\end{align*}
which completes the proof.

For the other direction, we note that all functions in $0\po AC$ are continuous, and let
  $f:\Om\to\R^l$ be with compact support and  not Lipschitz. Then for every $m\in \mathbb{N}$, there exists $\mathbf{x}_m,\mathbf{y}_m$ so that 
\begin{equation}\label{notLip}
\|f(\mathbf{x}_m)-f(\mathbf{y}_m)\|_{1} \geq m \|\mathbf{x}_m-\mathbf{y}_m\|_{1}.
\end{equation}

For each  $j=0,1,\dots, n,$ let
\begin{align*}
(\mathbf{z}_m^{(j)})_i = 
\begin{cases}
(\mathbf{x}_m)_i, i\leq j, \\
(\mathbf{y}_m)_i, i>j.
\end{cases}
\end{align*}

Then $\mathbf{z}_m^{(0)}=\mathbf{y}_m$, $\mathbf{z}_m^{(m)}=\mathbf{x}_m$, and 
$$
\|\mathbf{x}_m-\mathbf{y}_m\|_{1}=\sum\limits_{j=1}^n \|\mathbf{z}_m^{(j)}-\mathbf{z}_m^{(j-1)}\|_1.
$$
If for every $j=1,\dots, n$
$$
\|f(\mathbf{z}_m^{(j)})-f(\mathbf{z}_m^{(j-1)})\|_1 < m\|\mathbf{z}_m^{(j)}-\mathbf{z}_m^{(j-1)}\|_1,
$$
then we will have 
\begin{align*}
\|f(\mathbf{x}_m)-f(\mathbf{y}_m)\|_1 &\leq \sum\limits_{j=1}^n \|f(\mathbf{z}_m^{(j)})-f(\mathbf{z}_m^{(j-1)})\|_1 \\
&<m\cdot \sum\limits_{j=1}^n \|\mathbf{z}_m^{(j)}-\mathbf{z}_m^{(j-1)}\|_1 =m\cdot \|\mathbf{x}_m-\mathbf{y}_m\|_1.
\end{align*}
which contradicts \eqref{notLip}.
Thus there exists $j_m\in{1,\dots,n}$ so that 
\begin{equation} \label{goodnotLip}
 \|f(\mathbf{z}_m^{(j)})-f(\mathbf{z}_m^{(j-1)})\|_1 \geq m\cdot  \|\mathbf{z}_m^{(j)}-\mathbf{z}_m^{(j-1)}\|_1
\end{equation}
Note that all coordinates of $\mathbf{z}_m^{(j_m)} $ and $\mathbf{z}_m^{(j_m-1)}$ coincide, with the exception of the $j_m$-th coordinate. 

Since $f$ is a continuous function with compact support, $f$ is bounded and uniformly continuous and thus \eqref{goodnotLip} implies that 
\begin{equation*}\label{fzn}
\lim_{m\rightarrow\infty}\|f(\mathbf{z}_m^{(j_m)})-f(\mathbf{z}_m^{(j_m-1)})\|_1 = 0
\end{equation*}
and
\begin{equation*}\label{zn}
\lim_{m\rightarrow\infty} \|\mathbf{z}_m^{(j_m)}-\mathbf{z}_m^{(j_m-1)}\|_1 = 0.
\end{equation*}

We claim that $f\notin 0\po AC$.
Indeed, let $\e=\frac{1}{3^{n+1}}$,  $\de>0$ and $m\in \mathbb{N}$ so that $\frac{1}{m} < \de^{\frac{1}{n}}$. 
Then, by \eqref{goodnotLip}, we have 
$$
\frac{\|\mathbf{z}_m^{(j_m)}-\mathbf{z}_m^{(j_m-1)}\|_1^n}{\|f(\mathbf{z}_m^{(j_m)})-f(\mathbf{z}_m^{(j_m-1)})\|_1^n} <\de.
$$
Denote 
$
\gamma=\|f(\mathbf{z}_m^{(j_m)})-f(\mathbf{z}_m^{(j_m-1)})\|_1,
$
and choose $\m\in\mathbb{N}$ so that 
\begin{equation}\label{m}
\frac{1}{3}\gamma^{-n}\leq \m \leq \gamma^{-n}
\end{equation}
Then 
\begin{equation}\label{mzn}
\m\cdot \|\mathbf{z}_m^{(j_m)}-\mathbf{z}_m^{(j_m-1)}\|_1^n <\de
\end{equation}
Since $f$ is uniformly continuous, there exists $0<\eta<\|\mathbf{z}_m^{(j_m)}-\mathbf{z}_m^{(j_m-1)}\|_1$ so that for all
$\mathbf{x}, \mathbf{y}$  
\begin{equation}\label{eta}
\|\mathbf{x}-\mathbf{y}\|_1<\eta \Rightarrow \|f(\mathbf{x})-f(\mathbf{y})\|_1<\frac{1}{3}\cdot\gamma
\end{equation}
For $i=1,\dots,\m$, we define 
\begin{align*}
&\mathbf{a}_i=\mathbf{z}_m^{(j_m)}+i\Big(\frac{\eta}{n(\m+1)}\Big)\cdot\Big(\sum\limits_{v\neq j_m} \mathbf{e}_\nu\Big) \\
&\mathbf{b}_i=\mathbf{z}_m^{(j_m-1)}+(i+1)\Big(\frac{\eta}{n(\m+1)}\Big)\cdot\Big(\sum\limits_{v\neq j_m} \mathbf{e}_\nu\Big).
\end{align*}
Then for every $i$ we have 
\begin{align*}
&\max_\nu|\mathbf{a}_{i,\nu}-\mathbf{b}_{i,\nu}|= \|\mathbf{z}_m^{(j_m)}-\mathbf{z}_m^{(j_m-1)}\|_1, \\
&\|\mathbf{a}_i-\mathbf{z}_m^{(j_m)}\|_1=i\Big(\frac{\eta}{n(\m+1)}\Big)\cdot(n-1)<\eta, \\
&\|\mathbf{b}_i-\mathbf{z}_m^{(j_m-1)}\|_1=(i+1)\Big(\frac{\eta}{n(\m+1)}\Big)\cdot(n-1)<\eta.
\end{align*}
Thus, by \eqref{eta} we get 
\begin{align*}
\|f(\mathbf{a}_i)-f(\mathbf{b}_i)\|_1 &\geq \|f(\mathbf{z}_m^{(j_m)})-f(\mathbf{z}_m^{(j_m-1)})\|_1 - \frac{2}{3}\gamma \\
&= \frac{1}{3} \|f(\mathbf{z}_m^{(j_m)})-f(\mathbf{z}_m^{(j_m-1)})\|_1
\end{align*}
Hence, by \eqref{m} we get 
$$
\sum\limits_{i=1}^\m \|f(\mathbf{a}_i)-f(\mathbf{b}_i)\|_1^n \geq \m\cdot\frac{1}{3^n}\gamma^n\geq \frac{1}{3}\cdot\frac{1}{3^n}\geq \frac{1}{3^{n+1}}.
$$
On the other hand, by \eqref{mzn} we have 
\begin{align*}
 \sum\limits_{i=1}^\m (\max_\nu|\mathbf{a}_{i,\nu}-\mathbf{b}_{i,\nu}|)^n =
 \sum\limits_{i=1}^\m \|\mathbf{z}_m^{(j_m)}-\mathbf{z}_m^{(j_m-1)}\|_1^n <\de
\end{align*}
which ends the proof that $f\notin 0\po AC$.
\end{proof}

\section{The Hencl type extension of the class $1\po AC$}

We give an analog of Definition~\ref{qac}, and we prove that, similarly as for other classes of absolutely continuous functions, the classes $1\po AC_\la$ do not depend on 
$\la$ when $0<\la<1$.

Following \cite{Bongiorno2009}, we will use the following notation. 
Given interval $[\mathbf{x},\mathbf{y}]$, we denote $|f([\mathbf{x},\mathbf{y}])|=|f(\mathbf{y})-f(\mathbf{x})|$, and given $0<\la<1$, we denote by $~^\la[\mathbf{x},\mathbf{y}]$ the interval with center $(\mathbf{x}+\mathbf{y})/2$ and sides of length $\la(y_\nu-x_\nu)$, $\nu=1,\dots,n$.

\begin{defn} \label{1AClambda} (cf. \cite{Bongiorno2009})
 Let  $\al\in(0,1]$ and $\la\in(0,1)$. A function $f:\Om\to\R^l$ ($\Om\subset\R^n$ open)  is said to be in $\al\po AC_\la^{(n)}(\Om,\R^l)$ (briefly $\al\po AC_\la$) if  for each $\e>0$, there exists $\de>0$ such that for any finite collection of disjoint $\al$-regular intervals 
$\{ [\mathbf{a}_i,\mathbf{b}_i]\subset\Om \}_{i=1}^k$ 
we have 
\begin{equation*}\lb{1acla}
\sum_{i=1}^k \LL^n ([\mathbf{a}_i,\mathbf{b}_i])<\de \Rightarrow\sum_{i=1}^k|f(^\la[\mathbf{a}_i,\mathbf{b}_i)])|^n<\e.
\end{equation*}
\end{defn}

Bongiorno \cite{Bongiorno2009} proved that  for all $\al<1$, the class $\al\po AC_\la$ is independent of $\la$. We prove the same result for  $1\po AC_\la$.

\begin{thm}
\label{lambdaeq}
Let $0<\la_1<\la_2<1$. Then 
$$1\po  AC_{\la_1}^{(n)}(\Om,\R^l)=1\po AC_{\la_2}^{(n)}(\Om,\R^l).$$
\end{thm}

\begin{proof}
It is easy to see that $1\po\ACq_{\la_2} \subseteq 1\po\ACq_{\la_1}$. For the other direction, suppose $f \in 1\po\ACq_{\la_1}$.
 Fix $p\in \mathbb{N}$ such that $p>\frac{2\la_2}{(1-\la_2)\la_1}$. For any $\e>0$ there exists $\de>0$ such that for each finite family of non-overlapping 1-regular intervals $\{[\overline{\mathbf{a}_i},\overline{\mathbf{b}_i]}\subset \Om\}$, we have 
\begin{equation}
\label{la1cond}
\sum_{i}^{} \LL^n([\overline{\mathbf{a}_i},\overline{\mathbf{b}_i]})<\de \Rightarrow \sum_{i}^{}| f(^{\la_1}[\overline{\mathbf{a}_i},\overline{\mathbf{b}_i]})|^n<\frac{\e}{p^n}.
\end{equation}

Let $\{[\mathbf{a}_i,\mathbf{b}_i]\}$ be a finite family of  non-overlapping 1-regular intervals in $\Om$ with $\sum_{i}^{} \LL^n([\mathbf{a}_i,\mathbf{b}_i])<\de$. Let $\mathbf{c}_i$, $\mathbf{d}_i$ be such that $[\mathbf{c}_i,\mathbf{d}_i]= ~^{\la_2}[\mathbf{a}_i,\mathbf{b}_i]$. Then 
\begin{align*}
|f(^{\la_2}[\mathbf{a}_i,\mathbf{b}_i])| 
&= |f([\mathbf{c}_i,\mathbf{d}_i])| \\
&\leq \sum\limits_{j=0}^{p-1} \left|f\Big(\big[\mathbf{c}_i+\frac{(\mathbf{d}_i-\mathbf{c}_i)j}{p},\mathbf{c}_i+\frac{(\mathbf{d}_i-\mathbf{c}_i)(j+1)}{p}\big]\Big)\right|
\end{align*}
Hence there exists $j_0 \in \{0,\dots,p-1\}$ such that for 
$$
[\overline{\mathbf{a}_i},\overline{\mathbf{b}_i}]=^{~^{\frac{1}{\la_1}}}\left[\mathbf{c}_i+\frac{(\mathbf{d}_i-\mathbf{c}_i)j_0}{p},\mathbf{c}_i+\frac{(\mathbf{d}_i-\mathbf{c}_i)(j_0+1)}{p}\right],
$$
we have 
\begin{equation}
\label{la2leqla1}
|f(^{\la_2}[\mathbf{a}_i,\mathbf{b}_i])| \leq p|f(^{\la_1}[\overline{\mathbf{a}_i},\overline{\mathbf{b}_i}])|.
\end{equation}
Since $p>\frac{2\la_2}{(1-\la_2)\la_1}$, we obtain $[\overline{\mathbf{a}_i},\overline{\mathbf{b}_i}] \subseteq [\mathbf{a}_i,\mathbf{b}_i]$ and hence the intervals $[\overline{\mathbf{a}_i},\overline{\mathbf{b}_i}]$ are pairwise disjoint. Since 
$
\sum\limits_{i}^{} \LL^n([\overline{\mathbf{a}_i},\overline{\mathbf{b}_i}]) \leq \sum\limits_{i}^{} \LL^n([\mathbf{a}_i,\mathbf{b}_i])<\de,
$
by \eqref{la2leqla1} and \eqref{la1cond}, we get
$$
\sum\limits_{i}^{} |f(^{\la_2}[\mathbf{a}_i,\mathbf{b}_i])|^n \leq p^n\sum\limits_{i}^{} |f(^{\la_1}[\overline{\mathbf{a}_i},\overline{\mathbf{b}_i}])|^n < p^n\frac{\e}{p^n} = \e.
$$
\end{proof}

\begin{rem}\lb{def1ACH}
By Theorem~\ref{lambdaeq}, in analogy with Definition~\ref{qac} and \cite{Bongiorno2009} we will use the notation $1\po\ACq_H$ and $1\po\ACH$ instead of $1\po\ACq_{\la}$ and $1\po\ACq_{\la}^{(n)}(\Om,\R^l)$.
\end{rem}

\begin{cor}
$$\ACH \subseteq 1\po\ACq_{H}^{(n)}(\Om,\R^l).$$ 
\end{cor}
\begin{proof}
Note that it follows from the definition of $\al$-regularity of intervals that if $\al<\be$ and $f\in  \al\po AC_\la$ then $f\in  \be\po AC_\la$. In particular, for all $\al<1$,
$ \al\po AC_\la\subseteq1\po AC_\la$. Bongiorno \cite{Bongiorno2009} proved that  for all $\al<1$,  $\al\po AC_H=AC_H$. Thus $ AC_H\subseteq1\po AC_H$.
\end{proof}

\section{Class $1\po\ACq$}
\label{oneacprop}

To simplify notation, the  results of this section,  except Theorem \ref{diff11}, are stated for functions defined on subsets of $\R^2$ with range in $\R$. However they can be easily generalized to functions from
$\Om\subset \R^n$ to $\R^l$ for any $n\ge 2,l\in\N$.

We start from a structural result which will allow us to give examples of  functions in $1\po AC$.

\begin{thm} \label{lipone}
Let $d>0$ and let $S_d$ denote the square with vertices $(\pm{2}d,0)$, $(0,\pm{2}d)$. Let   $h:\R\lra \R$ and  $g:\R\lra \R$ be nonzero measurable functions with support contained in $[-d,d]$. We consider $\R^2$ with the basis $\{{\bf x}_1=(-1,1),{\bf x}_2=(1,1)\}$ and define $f : \R^{2} \longrightarrow \R$  with support contained in $S_d$ by 
$$f(s{\bf x}_1+t{\bf x}_2)=h(s)g(t).$$ 
(Alternatively, in the standard basis of $\R^2$, 
$f(x,y)=h(\frac{y-x}2)g(\frac{y+x}2).)$
\begin{itemize}
\item[(a)] If $h$ is bounded and $g$ is  Lipschitz, then $f$ is in $1$-$AC^{2}(\R^{2},\R)$.
\item[(b)] If $f\in 1\po\ACq_H^2(\R^2,\R) $, then $g$ is Lipschitz. 
\end{itemize}
\end{thm}

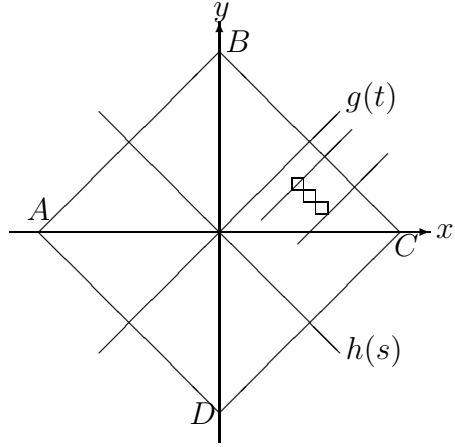
\setlength{\unitlength}{0.8cm}
\begin{picture}(7,8)
\put(0.5,4){\vector(1,0){7}}
\put(7.6,3.9){$x$}
\put(4,0.5){\vector(0,1){7}}
\put(3.9,7.6){$y$}
\linethickness{.025mm}
\put(1,4){\line(1,1){3}}
\put(7,4){\line(-1,-1){3}}
\put(1,4){\line(1,-1){3}}
\put(7,4){\line(-1,1){3}}
\put(2,6){\line(1,-1){4}}
\put(2,2){\line(1,1){4}}
\put(6.1,6.1){$g(t)$}
\put(6.1,1.9){$h(s)$}
\put(0.8,4.2){$A$}
\put(4.1,7){$B$}
\put(6.9,3.6){$C$}
\put(3.5,0.8){$D$}
\put(5.2,4.7){\line(1,0){0.2}}
\put(5.2,4.7){\line(0,1){0.2}}
\put(5.2,4.9){\line(1,0){0.2}}
\put(5.4,4.7){\line(0,1){0.2}}
\put(5.4,4.7){\line(1,0){0.2}}
\put(5.4,4.7){\line(0,-1){0.2}}
\put(5.4,4.5){\line(1,0){0.2}}
\put(5.6,4.5){\line(0,1){0.2}}
\put(5.6,4.5){\line(1,0){0.2}}
\put(5.6,4.5){\line(0,-1){0.2}}
\put(5.6,4.3){\line(1,0){0.2}}
\put(5.8,4.3){\line(0,1){0.2}}
\put(4.7,4.2){\line(1,1){1.5}}
\put(5.3,3.8){\line(1,1){1.5}}
\end{picture}

\begin{proof}[Proof of (a)]
Let $M\in \R$ be a bound of $h(s)$ and $L\in \R$ be a Lipschitz constant for $g(t)$. 
Then for all $s, t_1, t_2 \in [-d,d]$ we have 
\begin{equation}\lb{lipone1}
|f(s{\bf x}_1+t_{1}{\bf x}_2)-f(s{\bf x}_1+t_{2}{\bf x}_2)| = |h(s)|\cdot|g(t_1)-g(t_2)|\leq ML|t_1-t_2|.
\end{equation}

Let $\e>0$. Put $\de = \frac{\e}{M^2L^2}$.  Let  $\{[\mathbf{a}_i,\mathbf{b}_i]\subset S_d \}_{i=1}^k $ be any finite collection of disjoint 1-regular intervals  with $\sum_{i}^k \LL^2 ([\mathbf{a}_i,\mathbf{b}_i])<\de$. By 1-regularity of intervals $[\mathbf{a}_i,\mathbf{b}_i]$, there exist $s_i, t_{i,1},t_{i,2}\in[-d,d]$ so that  $\mathbf{a}_i = s_i{\bf x}_1+t_{i,1}{\bf x}_2$, $\mathbf{b}_i= s_i{\bf x}_1+t_{i,2}{\bf x}_2$ and
$\LL^2 ([\mathbf{a}_i,\mathbf{b}_i]= |t_{i,1}-t_{i,2}|^2.$
Thus, by \eqref{lipone1}, we have
\begin{equation*}
\begin{split}
 \sum_{i=1}^{k} |f(\mathbf{a}_i)-f(\mathbf{b}_i)|^2 &\le\sum_{i=1}^{k} (ML|t_{i,1}-t_{i,2}|)^2\\
&=M^2L^2 \sum_{i=1}^{k} \LL^2 ([\mathbf{a}_i,\mathbf{b}_i])<M^2L^2\cdot\de=\e.
\end{split}
\end{equation*}
So $f\in 1$-$AC^{2}(\R^{2},\R)$.
\end{proof}

\begin{proof}[Proof of (b)]
Let $I\subseteq[-d,d]$ be  a set of positive measure so that there exists $c>0$ such that $ |h(s)|\ge\frac{1}{c}$ for all $s\in I$. Let $\e=1$ and $\la \in(0,1)$. Since $f\in 1\po\ACq_H$, there exists $\de>0$ so that for all finite families of disjoint 1-regular intervals $\{[\mathbf{a}_i,\mathbf{b}_i]\}$ we have 
\begin{equation}
\label{est}
\sum_{i} \LL^n([\mathbf{a}_i,\mathbf{b}_i])<\de\Rightarrow\sum_{i}  |f(^{\la}[\mathbf{a}_i,\mathbf{b}_i])|^2<1.
\end{equation} 

\Buo $\sqrt{\de}<\LL^1(I)$. Let $t,t'\in [-d,d]$ be such that $t<t'$ and $|t'-t|<\la\sqrt{\de}$, and let $k\in \mathbb{N} $ be  such that 
\begin{equation}
\label{sizek}
\la\sqrt{\frac{\de}{k+1}}.\leq|t-t'|<\la\sqrt{\frac{\de}{k}}.
\end{equation}

Let $\{s_i\}_{i=1}^k\subset I$ be such that $|s_i-s_j|>\sqrt{\frac{\de}{k}}$ for all $i\neq j$. 

Put 
\begin{equation*}
\begin{split}
\mathbf{a}_i&=s_i{\bf x}_1+\Big(\frac{t+t'}{2}-\frac{1}{\la}\frac{|t-t'|}{2}\Big){\bf x}_2,\\
\mathbf{b}_i&=s_i{\bf x}_1+\Big(\frac{t+t'}{2}+\frac{1}{\la}\frac{|t-t'|}{2}\Big){\bf x}_2.
\end{split}
\end{equation*}

Then $\{[\mathbf{a}_i,\mathbf{b}_i]\}_{i=1}^k$ is a family of disjoint 1-regular intervals  with
$
\LL^2[\mathbf{a}_i,\mathbf{b}_i]= (\frac{1}{\la}|t-t'|)^2<\frac{\de}{k},
$
and $^{\la}[\mathbf{a}_i,\mathbf{b}_i]=[s_i{\bf x}_1+t{\bf x}_2,s_i{\bf x}_1+t'{\bf x}_2]$, for all $i$.
Thus, by \eqref{est}, we have 
\begin{align*}
1&>\sum_{i=1}^k |f(s_i{\bf x}_1+t{\bf x}_2)-f(s_i{\bf x}_1+t'{\bf x}_2)|^2\\
&=\sum_{i=1}^k |h(s_i)|^2|g(t)-g(t')|^2 \\
&\geq k\cdot\frac{1}{c^{2}} |g(t)-g(t')|^2.
\end{align*}
By \eqref{sizek} we get
$$
|g(t)-g(t')|^2 \leq \frac{c}{k} \leq \frac{c^{2}}{\de\la^2} |t-t'|^2.
$$

Thus $g$ is continuous, and therefore bounded on $[-d,d]$, and $g$ is Lipschitz with Lipschitz constant $L=\max(\sqrt{\frac{c^{2}}{\de\la^{2}}}, \frac{2M}{\la\sqrt{\de}})$, where $M=\max\{|g(t)|:t\in [-d,d]\}$. 
\end{proof}

As a consequence of Theorem~\ref{lipone} we obtain that, on the one hand, not every differentiable function belongs to $1\po AC_H^{2}(\R^{2},\R)$ and that $W^{1,2}_{loc}(\R^2)\not\subset1\po AC_H^{2}(\R^{2},\R)$, and on the other hand $1\po\ACq^{2}(\R^{2},\R)$ contains examples of several types of functions with
pathological properties.

\begin{cor}\lb{diffnot1ac}
(a) There exists a function differentiable everywhere which does not belong to $1\po\ACq_H$, and 
(b)  $W^{1,2}_{loc}(\R^2)\not\subset1\po AC_H^{2}(\R^{2},\R)$.
\end{cor}

\begin{proof}
It is enough to take a function like in Theorem~\ref{lipone} with $h$ constant and, for (a), $g$ differentiable but not Lipschitz, e.g. $g(t)=\sqrt[3]{t} $ on $[-\frac{d}{2},\frac{d}{2}]$, for (b), $g$ in  $W^{1,2}(\R)$ but not Lipschitz.
\end{proof}

\begin{cor}\lb{exbdddisc}
There exists a function in $1\po\ACq$ which is 
bounded and discontinuous everywhere on its support.
\end{cor}

\begin{proof}
It is enough to take a function like in Theorem~\ref{lipone} with $g$ constant and $h$ bounded but discontinuous everywhere.
\end{proof}

\begin{cor}\lb{exunbdd}
There exists a function in $1\po\ACq$ which is supported on a compact set and
unbounded.
\end{cor}

\begin{proof}
An example is provided by  a function like in Theorem~\ref{lipone} with $h$ unbounded on $[-d,d]$ and $g(t)=1$ for $t\in [-d,d]$. 

The sum of  $f$ and any function  in $1\po\ACq$ which is not constant on segments with slope 1 gives an example of an unbounded function in $1\po\ACq$ which is not constant on segments with slope 1.
\end{proof}

\begin{cor}\lb{nonSobolev}
There exists a function in $1\po\ACq$ which is differentiable everywhere but is not in the  Sobolev space $W^{1,2}_{loc}(\R^2)$.
\end{cor}

\begin{proof}
An example is provided by a function like in Theorem~\ref{lipone} with  $g$ constant and $h\in W^{1,1}(-1,1)\setminus W^{1,2}(-1,1)$. 
\end{proof}

\begin{cor} \label{oneacdiff}
There exists a continuous function $f\in 1\po\ACq^{2}(\R^{2},\R)$ such that for every $\mathbf{p}\in\supp f$ and every direction ${\bf v}\neq (1,1)$ the directional derivative $D_{\bf v} f(\mathbf{p})$ does not exist.  In particular $f$ is not differentiable anywhere on its support.
\end{cor}

\begin{proof}
For the proof we will use the Takagi function $T$  which is continuous on $[0,1]$ but is nowhere differentiable (\cite{Takagi1903}, cf.   \cite[p.~36]{Thim2003}) and which is defined by
$$
T(x)=\sum\limits_{k=0}^\infty \frac{1}{2^k} \dist (2^kx,\mathbb{Z})=\sum\limits_{k=0}^\infty \frac{1}{2^k} \inf\limits_{m\in\mathbb{Z}} |2^kx-m|.
$$  
Let   $g:[-\frac12,\frac12]\to \R$ be defined by $g(t)=1-2|t|$.
Define $f : \R^2 \rightarrow \R$  by $$f(s{\bf x}_1+t{\bf x}_2)=T\Big(s+\frac12\Big)g(t).$$ 
Since $T(s)$ is bounded and $g$ is Lipschitz, by Theorem~\ref{lipone}(a),  $f\in 1\po\ACq$. 

We claim that $f$ is not differentiable anywhere. Indeed, for any point $\mathbf{p}=s_0{\bf x}_1+t_0{\bf x}_2\in \supp f$, for any direction $\mathbf{v}=v_1{\bf x}_1+v_2{\bf x}_2$ other than ${\bf x}_2=(1,1)$, we have 
\begin{equation*}
\lim_{h\to 0} \frac{f(\mathbf{p}+h\mathbf{v})-f(\mathbf{p})}{h} 
= \lim_{h\to0} \frac{T(s_0+\frac12+hv_1)g(t_0+hv_2)-T(s_0+\frac12)g(t_0)}{h}.
\end{equation*}

Since $g(t_{0}+hv_2)=1-2(|t_0|\pm hv_2)$
for $\left|h\right|<\left|t_{0}\right|/v_{2}$, $g(t_0)=1-2|t_0|,
$
and $T(s)$ is continuous,  
we get
\begin{align*}
&\lim_{h\to 0} \frac{(1-2(|t_0+hv_2|)T(s_0+\frac12+hv_1)-(1-2|t_0|)T(s_0+\frac12)}{h} \\
&= \lim_{h\rightarrow0} (1-2|t_0|)\Big[\frac{T(s_0+\frac12+hv_1)-T(s_0+\frac12)}{h}\Big]\mp\lim_{h\rightarrow0} 2v_2T(s_0+\frac12+hv_1)\\
&= \lim_{h\rightarrow0}^{} (1-2|t_0|)\Big[\frac{T(s_0+\frac12+hv_1)- T(s_0+\frac12)}{h}\Big]\mp 2v_2T(s_0+\frac12).
\end{align*}
But, since $v_1\ne 0$ and the Takagi function $T$ is nowhere differentiable,  $$\lim_{h\rightarrow0} \frac{T(s_0+\frac12+hv_1)-T(s_0+\frac12)}{h} $$ does not exist  anywhere. Therefore $D_{\mathbf{v}}f$ does not exist anywhere.
\end{proof}

Next we study  the Luzin (N) property for differentiable functions in $1\mbox{-}AC(\R^{n},\R^{l})$ where $n>1$. Recall that a function $f:\R^{n}\to\R^{l}$ is said to have {\it the Luzin (N) property} if $\mathcal{H}^{n}(f(E))=0$ whenever $E\subseteq \R^{n}$ and $\leb^{n}(E)=0$, where $l\geq n$ and  $\mathcal{H}^{n}$ denotes the $n$-dimensional Hausdorff measure on $\R^{l}$, see e.g. \cite{Ziemer}. 

It is  known that all absolutely continuous functions on $\R$ and all functions in $AC_H$ satisfy the Luzin~(N) condition. However, this property is not guaranteed for differentiable functions in $1\mbox{-}AC(\R^{n},\R^{l})$ where $n>1$.
\begin{thm}\lb{luzin}
Suppose $n>1$ and $l\geq 1$ are integers. Then there exists a differentiable function $f$ in $1\po AC(\R^{n},\R^{l})$ and a set $U\subseteq\R^{n}$ with $\leb^{n}(U)=0$ and $\leb^{l}(f(U))>0$.
\end{thm}
\begin{remark}
Note that when $l\geq n$, any subset of $\R^{l}$ with positive $l$-dimensional Lebesgue measure, has positive $n$-dimensional Hausdorff measure. In fact, if $l>n$, the $n$-dimensional Hausdorff measure of such a set is necessarily infinite. Hence, Theorem~\ref{luzin} shows that functions in the class $1\po AC_{D}(\R^{n},\R^{l})$ may be severely expanding; when $l\geq n$, the conclusion of Theorem~\ref{luzin} is stronger than the assertion that $f$ does not have the Luzin (N) property.
\end{remark}

\begin{proof}[Proof of Theorem~\ref{luzin}]
Let $\varphi:\R\to[0,1]$ be an extension of the Cantor function (see \cite{Cantor84}, \cite{DMRV2006}) on $[0,1]$ such that $\varphi$ is constant on $\R\setminus[0,1]$. Let $C$ denote the standard ternary Cantor set. Recall that   $\varphi(C)=[0,1]$ and  $\varphi$ is constant on each connected component of $[0,1]\setminus C$. Hence, $\varphi$ is differentiable with derivative zero almost everywhere.

Let $p:[0,1]\to\R^{l-1}$ denote a space filling curve  with $\leb^{l-1}(p([0,1])) >0$, (see \cite{ACQ98}). Note that the function $p\circ \varphi:[0,1]\to\R^{l-1}$ is differentiable with derivative zero almost everywhere.

Let $\mathbf{x}_{1},\ldots,\mathbf{x}_{n}$ of $\R^{n}$ be an orthonormal basis  of $\R^{n}$ so that $\mathbf{x}_{n}=(1/\sqrt{n})(\mathbf{e}_{1}+\ldots+\mathbf{e}_{n})$. Define a set $U\subseteq\R^{n}$ by
\begin{equation*}
U=\left\{t_{1}\mathbf{x}_{1}+\ldots+t_{n}\mathbf{x}_{n}\mbox{ : }t_{1}\in C,t_{2}\ldots,t_{n}\in\R \right\}.
\end{equation*}
We note that $U$ has Lebesgue measure zero, since it is an isomorphic image of $C\times\R^{n-1}$. Let $f:\R^{n}\to\R^{l}$ be the function defined by 
\begin{equation*}
f(t_{1}\mathbf{x}_{1}+\ldots+t_{n}\mathbf{x}_{n})=\begin{cases}
\varphi(t_{1}) & \mbox{ if }l=1,\\
(p\circ\varphi(t_{1}),t_{2}) & \mbox{ if }l>1.
\end{cases} 
\end{equation*}
Then
\begin{equation*}
f(U)=\begin{cases}
[0,1] & \mbox{ if }l=1,\\
p([0,1])\times \R & \mbox{ if }l>1.
\end{cases}
\end{equation*}
Hence, $\leb^{l}(f(U))>0$, and $f$ does not have the Luzin (N) property. Further, note that $f$ is differentiable almost everywhere.

It only remains to verify that $f\in 1\po AC(\R^{n},\R^{l})$. One can check that each component of the function $f$ has the form given by the generalization of Theorem~\ref{lipone}, part (a) (for $\R^{n}$ rather than $\R^{2}$). By an adaptation of the proof of Theorem~\ref{lipone}, part (a), we get that $f\in 1\po AC(\R^{n},\R^{l})$. 
\end{proof}

The next theorem contains a positive result about properties of functions in $1\po AC_{H}$.

\begin{thm}\lb{diff11} Every function 
$
f\in  1\po AC_{H}^{(n)}(\R^{n},\R^{l})
$
is differentiable a.e. in the direction ${\bf e}_1+{\bf e}_2+\dots +{\bf e}_n$.
\end{thm}
\begin{proof}
Fix  $f\in1\po AC_{H}^{(n)}(\R^{n},\R^{l})$.  If $n=1$ then  $f=(f_{1},\ldots,f_{l})$ where each $f_{i}:\R\to\R$ is a function in $AC_{H}$. Thus each $f_{i}$ is differentiable almost everywhere, and the theorem follows.

Let $n>1$. Let $\mathbf{x}_{n}={\bf e}_1+{\bf e}_2+\dots +{\bf e}_n$ and choose vectors $\mathbf{x}_{1},\ldots,\mathbf{x}_{n-1}\in\R^{n}$ so that $\mathbf{x}_{1},\ldots,\mathbf{x}_{n}$ is an orthogonal basis. In what follows we will identify $\R^{n-1}$ with the $n-1$ dimensional subspace of $\R^n$ spanned by $\mathbf{x}_{1},\ldots,\mathbf{x}_{n-1}$ via the correspondence ${\bf~s}\leftrightarrow s_{1}\mathbf{x}_{1}+\ldots +s_{n-1}\mathbf{x}_{n-1}$. For $\mathbf{u}\in\R^{n}$, we define a function $f_{\mathbf{u}}:\R\to\R^{l}$ by $f_{\mathbf{u}}(t)=f(\mathbf{u}+t\mathbf{x}_{n})$. Moreover, given a point $t_{0}\in\R$ and a function $g:\R\to\R^{l}$ we define
\begin{equation}\lb{lip}
\lip(g,t_{0})=\limsup_{t\to t_{0}}\frac{|g(t)-g(t_{0})|}{|t-t_{0}|}.
\end{equation}
We will show that the set
\begin{equation*}
E:=\left\{\mathbf{u}\in\R^{n}\mbox{ : }\lip(f_{\mathbf{u}},0)=\infty\right\}
\end{equation*} 
has $n$-dimensional Lebesgue measure zero.

We claim that this suffices: Indeed, if $E$ has measure zero then, using Fubini's Theorem, we get that for almost every ${\bf s}\in\R^{n-1}$, $\lip(f_{{\bf s}+t\mathbf{x}_{n}},0)<\infty$ for almost every $t\in\R$. Observe that $\lip(f_{{\bf s}},t)=\lip(f_{{\bf s}+t\mathbf{x}_{n}},0)$ for all ${\bf s}\in\R^{n-1}$ and $t\in\R$. It follows that $\lip(f_{{\bf s}},t)<\infty$ for almost every $s\in\R^{n-1}$ and almost every $t\in\R$. Now, applying the Stepanov Theorem \cite{ZBZ2003}, we  conclude that for almost every ${\bf s}\in\R^{n-1}$, $f_{{\bf s}}$ is differentiable almost everywhere (in $\R$). Clearly $f$ is differentiable at ${\bf s}+t\mathbf{x}_{n}$ in the direction $\mathbf{x}_{n}$  if and only if $f_{{\bf s}}$ is differentiable at $t$. Hence $f$ is differentiable in the direction $\mathbf{x}_{n}$ almost everywhere.

We now prove that  $E$ has measure zero. Note that $E$ is  measurable.  
Fix $\e>0$ and choose $\delta>0$ such that
\begin{equation}\lb{eq:1accond}
\sum_{k=1}^{N}\leb^{n}([{\bf a}_{k},{\bf b}_{k}])<\delta\Rightarrow\sum_{k=1}^{N}\Big|f\Big(\frac{{\bf a}_{k}+3{\bf b}_{k}}{4}\Big)-f\Big(\frac{3{\bf a}_{k}+{\bf b}_{k}}{4}\Big)\Big|^{n}<\e,
\end{equation}
whenever $\{[a_{k},b_{k}]\}_{k=1}^{N}$ is a finite collection of pairwise disjoint, $1$-regular intervals in $\R^{n}$. Let $\{A_{m}\}_{m=1}^{\infty}$ be a countable collection of closed intervals with pairwise disjoint, non-empty interiors such that $\R^{n}=\bigcup_{m=1}^{\infty}A_{m}$ and $\leb^{n}(A_{m})<\delta$ for all $m$. For each $m\in\N$, we define a collection of intervals $\mathcal{V}_{m}$ by
\begin{equation}\label{eq:Vm2}
\mathcal{V}_{m}=\Big\{[{\bf a},{\bf b}]\subseteq A_{m}\mbox{ : }\frac{\Big|f\big(\frac{{\bf a}+3{\bf b}}{4}\big)-f\big(\frac{3{\bf a}+{\bf b}}{4}\big)\Big|}{\frac{\|{\bf b}-{\bf a}\|}{2}}\geq m\Big\}.
\end{equation}
Note that $\mathcal{V}_{m}$ is a Vitali cover of $E\cap \inter(A_{m})$. Hence, by the Vitali Covering Theorem, we can find a collection $\{I_{k}^{(m)}=[{\bf a}_{k}^{(m)},{\bf b}_{k}^{(m)}]\}_{k=1}^{\infty}$ of pairwise disjoint intervals from $\mathcal{V}_{m}$ such that
\begin{equation*}
\leb^{n}\Big((E\cap\inter(A_{m}))\setminus\bigcup_{k=1}^{\infty}I_{k}^{(m)}\Big)=0.
\end{equation*}

Choose an integer $K_{m}\geq 1$ so that $\sum_{k=1}^{K_{m}}\leb^{n}(I_{k}^{(m)})>\leb^{n}(E\cap A_{m})-\frac{\e}{2^{m}}$.
Since the intervals $I_{k}^{(m)}$ are pairwise disjoint and contained in $A_{m}$ we have that $\sum_{k=1}^{K_{m}}\leb^{n}(I_{k}^{(m)})<\leb^{n}(A_{m})<\delta$. Thus, using \eqref{eq:1accond} and \eqref{eq:Vm2}, we get
\begin{align*}
\e&>\sum_{k=1}^{K_{m}}\Big|f\Big(\frac{\mathbf{a}_{k}^{(m)}+3\mathbf{b}_{k}^{(m)}}{4}\Big)-f\Big(\frac{3\mathbf{a}_{k}^{(m)}+\mathbf{ b}_{k}^{(m)}}{4}\Big)\Big|^{n}\\
&\geq 2^{-n}m^{n}\sum_{k=1}^{K_{m}}\big\|{\bf b}_{k}^{(m)}-{\bf a}_{k}^{(m)}\big\|^{n}\\
&\geq 2^{-n}m^{n}\sum_{k=1}^{K_{m}}\leb^{n}(I_{k}^{(m)}).
\end{align*}
Hence, 
\begin{equation*}\label{eq:sumbound}
\sum_{k=1}^{K_{m}}\leb^{n}(I_{k}^{(m)}\leq\frac{2^{n}\e}{m^{n}}.
\end{equation*}
We now deduce that
\begin{equation*}\label{eq:sumsquares}
\sum_{m=1}^{\infty}\sum_{k=1}^{K_{m}}\leb^{n}(I_{k}^{(m)})\leq 2^{n}\e\sum_{m=1}^{\infty}\frac{1}{m^{n}}<2^{n+1}\e ,
\end{equation*}
whilst
\begin{equation*}\label{eq:covermost}
\sum_{m=1}^{\infty}\sum_{k=1}^{K_{m}}\leb^{n}(I_{k}^{m})\geq\sum_{m=1}^{\infty}(\leb^{n}(E\cap A_{m})-\frac{\e}{2^{m}})\geq\leb^{n}(E)-\e.
\end{equation*}
Thus $\leb^{n}(E)\leq (2^{n+1}+1)\e$. Since $\e>0$ was arbitrary,  $\leb^{n}(E)=0$.
\end{proof}

In view of the presented above pathological examples of functions in $1\po AC$, it makes sense to consider the class $1\po AC_{\rm WDN}$ which  consits of all functions in $1\po AC$ which are in the Sobolev space $W^{1,n}_{loc}$,  are differentiable a.e. and satisfy the Luzin (N) property.

Our final result  shows where the class $1\po AC_{\rm WDN}$ fits in the hierarchy of previously studied classes.

\begin{thm}\lb{thmcontainments}
The following holds
\begin{equation}\label{containments}
\begin{split}
 &\mathcal{Q}\po\ACq \subsetneq \al\po\ACq= 1\po\ACq_{\rm WDN}\cap \ACq_H\subsetneq 1\po\ACq_{\rm WDN}\\& \subsetneq 1\po\ACq_{\rm WDN}\cup \ACq_H \subsetneq \linspan( 1\po\ACq_{\rm WDN}\cup \ACq_H)\subseteq1\po\ACq_{\rm HWDN},
\end{split}
\end{equation}
and 
\begin{equation}\label{bacnotone}
1\po\ACq_{\rm WDN}\setminus \mathcal{B}\po\ACq \ne\emptyset,  \ \ \ \text{ and }\ \ \ \mathcal{B}\po\ACq \setminus 1\po\ACq_{\rm WDN}\ne\emptyset.
\end{equation}
where $1\po\ACq_{\rm HWDN}$ denotes the set of functions in $1\po\ACq_H$ which are in the Sobolev space $W^{1,2}_{loc}$,  are differentiable a.e. and satisfy the Luzin (N) property.
\end{thm}

\begin{rem} \lb{open}
Bongiorno \cite{Bongiorno2009a} introduced another class of absolute continuity denoted $AC_\La^n(\Om,\R^l)$, or simply $AC_\La$, so that $AC_H\varsubsetneq AC_\La$ and all functions in $AC_\La$ are differentiable a.e. and satisfy the Luzin (N) property, but $AC^n_\La(\Om,\R^l)\not\subset W^{1,n}_{loc}(\Om,\R^l)$. 
It would be interesting to determine 
what are the  classes  $1\po AC \cap AC_\La$,  $1\po AC_H \cap AC_\La$ and  $1\po AC_{\rm HWDN} \cap AC_\La$.

It also would be interesting to determine what are the relations between $AC_\La\cap W^{1,n}_{loc}$, 
$1\po\ACq_{\rm HWDN}$ and the  linear span of $1\po\ACq_{\rm WDN}\cup \ACq_H$.
\end{rem}

\begin{proof}[Proof of Theorem~\ref{thmcontainments}]
The first containment of \eqref{containments} is due to Bongiorno \cite{Bongiorno2005}. The next equality follows from
Theorem~\ref{1capH}, since all functions in $\ACq_H$ are in the Sobolev space $W^{1,2}_{loc}$,  are differentiable a.e. and satisfy the Luzin (N) property.
All following containments are clear. The proof that
\begin{equation}\lb{1acnotach}
1\po\ACq_{\rm WDN}\setminus\ACq_H\ne\emptyset
\end{equation} 
is technical and we postpone it till the end.

The Bongiorno's example \cite[Example~4.(1)]{Bongiorno2005}  shows that 
\begin{equation}\lb{1achnot1ac}
\ACq_{H} \setminus 1\po AC_{\rm WDN} \ne\emptyset.
\end{equation} 

Since both $\ACq_{H}$ and $1\po AC_{\rm WDN} $ are linearly closed it follows from \eqref{1achnot1ac} that $\linspan( 1\po\ACq_{\rm WDN}\cup \ACq_H)\setminus (1\po\ACq_{\rm WDN}\cup \ACq_H)\ne\emptyset$.

The first part of  \eqref{bacnotone} follows from \eqref{1acnotach} since $\mathcal{B}\po\ACq\subset AC_H$.

The other part of \eqref{bacnotone} follows from   a small adjustment of the function constructed by  Cs{\"o}rnyei in \cite[Theorem~2]{Csornyei2000}.

Indeed, let $f$ be the function  defined in \cite[Theorem~2]{Csornyei2000}. We rotate $f$ clockwise by 90$^\circ$ to obtain the function $g$. Cs{\"o}rnyei showed that  $f\notin Q\po\ACq^{2}(\Om,\R)$. By a similar argument, using the same notation, since the right-upper corner of each $Q_{mk}=[\mathbf{a}_{mk},\mathbf{b}_{mk}]$ is a 1-regular interval for $m\in\mathbb{N}$ and $k=1,2,\dots,r_m$, and since the collection $\{Q_{mk}\}$ forms a pairwise disjoint system of 1-regular intervals,  we have for the new function $g$
 $$|g(\mathbf{b}_{mk})-g(\mathbf{a}_{mk})|=\om_m. $$ 
Thus, 
$$
\sum\limits_{m=1}^{\infty} \sum\limits_{k=1}^{r_m}|g(\mathbf{b}_{mk})-g(\mathbf{a}_{mk})|^2=\sum\limits_{m=1}^{\infty} \sum\limits_{k=1}^{r_m} \om^2_m = \sum\limits_{m=1}^{\infty} r_m\om_m^2= \sum\limits_{m=1}^{\infty} \frac{1}{4m}=\infty.
$$
Therefore, $g\notin 1\po\ACq^{2}(\Om,\R^l)$.
However, by the same argument as in  \cite[Theorem~2]{Csornyei2000}, $g$ is in $\mathcal{B}\po\ACq^{2}(\Om,\R)$.

We now prove \eqref{1acnotach}.
 The construction is an adjustment of  \cite[Theorem~2]{Csornyei2000} and  \cite[Example~4.(2)]{Bongiorno2005}. Since the construction is very technical we provide all details for the convenience of the reader.

We define $f$ as the sum of an absolutely convergent series of non-negative continuous functions $f_m$ so that the support of each $f_m$ is covered by the union of 
$$
r_m=4^{m-1}(m-1)!m!
$$
pairwise disjoint squares
$$
Q_{m1},Q_{m2},\dots,Q_{mr_m},
$$
with
\begin{equation} \lb{sizeQmk}
\LL^2\left(\bigcup_{k=1}^{r_{m+1}} Q_{(m+1)k}\right)<\frac18 \LL^2\left(\bigcup_{k=1}^{r_{m}} Q_{mk}\right)<\left(\frac18\right)^{m}\LL^2(Q_{11}),
\end{equation}
and such that 
$$
\max{f_m}=\om_m\DEF\frac{1}{2^mm!}.
$$

The square $Q_{11}$ is arbitrarily chosen in $\Om$. Assume that, for a given $m$, the functions $f_1,f_2,\dots,f_{m-1}$ and the squares $Q_{hk},1\leq h\leq m,1\leq k \leq r_m$, have been defined. We define $f_m$ and the squares $\mathcal{Q}_{(m+1)j}$, for $j=1,\dots,r_{m+1}$ as follows: for a fixed $k\in {1,\dots,r_m}$, we put a horizontal and a vertical line through the midpoint $O=(o_1,o_2)$ of the square $Q_{mk}$.  Denote by $2d_1$ the length of the side of $Q_{mk}$ and by  
\begin{equation*}
\begin{split}
A_1&=(o_1+d_1,o_2+d_1),\  A_2=(o_1+d_1,o_2-d_1),\\
A_3&=(o_1-d_1,o_2-d_1), \ A_4=(o_1-d_1,o_2+d_1),
\end{split}
\end{equation*}
the verices of the square $Q_{mk}$.

Let  $d=\frac12 d_1$. We construct a smaller square $A_1B_1C_1D_1$ in the upper right corner of the square $A_1A_2A_3A_4$, where 
\begin{equation}\lb{defupper}
B_1=(o_1+d_1,o_2+d),C_1=(o_1+d,o_2+d),D_1=(o_1+d_1,o_2+d).
\end{equation}
The length of side $A_1B_1=d_1-d=d$. Note that the interval $T_{mk}=[O,B_1]\subset Q_{mk}$ is $\frac12$-regular  since 
\begin{equation}\lb{12regular}
\frac{\LL^2(T_{mk})}{d_1^2}=\frac{\LL^2[O,B_1]}{d_1^2}=\frac{d\cdot d_1}{d_1^2}=\frac12. 
\end{equation}

Similarly we construct a square $A_3B_3C_3D_3$ in the lower left corner of the square $A_1A_2A_3A_4$, where 
$$B_3=(o_1-d_1,o_2-d),C_3=(o_1-d,o_2-d),D_3=(o_1-d,o_2-d_1).$$

Look at the square $A_1B_1C_1D_1$ with the side length equal to $d$, cf. Figure~\ref{crazyfig}. For every interval 
\begin{equation*}
[a_i,b_i]=\begin{cases}
[\frac{d}{2^i},\frac{d}{2^{i-1}}], i=1,2,...,m-1,\\
[0,\frac{d}{2^{m-1}}], i=m ,
\end{cases}
\end{equation*}
we put $2(m+1)$ small disjoint squares $	Q_{(m+1)j}$ inside each strip 
$$
M_i=\left\{(x,y)\in \R^2,\rho (x,y)\in \Big[\frac{2a_i+b_i}{3},\frac{a_i+2b_i}{3}\Big] \cap \square A_1B_1C_1D_1\right\},
$$
where $\rho (x,y) = |x-o_1-d|+|y-o_2-d|$ (as shaded in Figure~\ref{crazyfig});
and we also put $2(m+1)$ small disjoint squares $Q_{(m+1)j}$ inside each strip
$$
M_i'=\left\{(x,y)\in \R^2,\rho' (x,y)\in \Big[\frac{2a_i+b_i}{3},\frac{a_i+2b_i}{3}\Big] \cap \square A_3B_3C_3D_3\right\},
$$
where $\rho' (x,y) = |x-o_1+d|+|y-o_2+d|$.

Thus all squares $Q_{(m+1)j}$ are contained in the union of of squares $A_1B_1C_1D_1$ and $A_3B_3C_3D_3$, whose union has measure equal to $\frac18$ of the measure of square $Q_{mk}$, so that \eqref{sizeQmk} is satisfied.

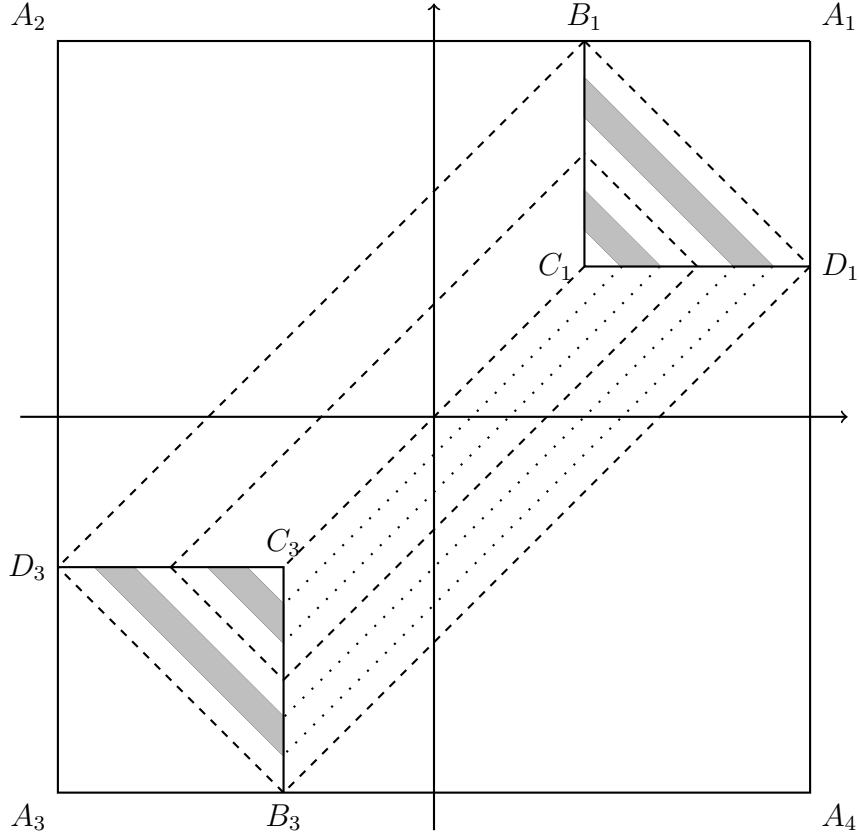
\begin{figure}
\begin{tikzpicture}[thick]
\draw [->](-5.5,0)--(5.5,0);
\draw [->](0,-5.5)--(0,5.5);
\coordinate [label=45:$A_1$] (A1) at (5,5);
\coordinate [label=135:$A_2$] (A2) at (-5,5);
\coordinate [label=-135:$A_3$] (A3) at (-5,-5);
\coordinate [label=-45:$A_4$] (A4) at (5,-5);
\coordinate [label=180:$C_1$] (C1) at (2,2);
\coordinate [label=90:$C_3$] (C3) at (-2,-2);
\coordinate [label=90:$B_1$] (B1) at (2,5);
\coordinate [label=-90:$B_3$] (B3) at (-2,-5);
\coordinate [label=0:$D_1$] (D1) at (5,2);
\coordinate [label=180:$D_3$] (D3) at (-5,-2);
\draw [dashed] (3.5,2)--(2,3.5);
\draw [dashed] (-3.5,-2)--(-2,-3.5);
\draw [dashed] (B1)--(D1);
\draw [dashed] (B3)--(D3);
\draw [dashed] (C1)--(C3);
\draw [dashed] (B1)--(D3);
\draw [dashed] (D1)--(B3);
\draw [dashed] (-2,-3.5)--(3.5,2);
\draw [dashed] (-3.5,-2)--(2,3.5);
\draw [loosely dotted] (2.5,2)--(-2,-2.5);
\draw [loosely dotted] (3,2)--(-2,-3);
\draw [loosely dotted] (4,2)--(-2,-4);
\draw [loosely dotted] (4.5,2)--(-2,-4.5);
\draw  (2.5,2)--(2,2.5);
\draw  (-2.5,-2)--(-2,-2.5);
\draw  (3,2)--(2,3);
\draw  (-3,-2)--(-2,-3);
\draw  (4,2)--(2,4);
\draw  (-4,-2)--(-2,-4);
\draw  (4.5,2)--(2,4.5);
\draw  (-4.5,-2)--(-2,-4.5);
\filldraw[lightgray] (2.5,2)--(2,2.5)--(2,3)--(3,2)--(2.5,2);
\filldraw[lightgray] (4,2)--(2,4)--(2,4.5)--(4.5,2)--(4,2);
\filldraw[lightgray] (-2.5,-2)--(-2,-2.5)--(-2,-3)--(-3,-2)--(-2.5,-2);
\filldraw[lightgray] (-4,-2)--(-2,-4)--(-2,-4.5)--(-4.5,-2)--(-4,-2);
\draw (A1)--(A2)--(A3)--(A4)--(A1);
\draw (B1)--(C1)--(D1);
\draw (B3)--(C3)--(D3);
\end{tikzpicture}
\caption{$Q_{mk}$ when $m=2$}
\label{crazyfig}
\end{figure}

 We require that the distribution of the squares inside the strips is the following (we will describe the situation for the strip $M_i$, and the strip $M_i'$ will be symmetric):

$\bf{Case\ 1}$: if $m$ is even, then we put
\begin{itemize}
\item[\textcircled{1}] $m/2+1$ squares $Q_{(m+1)j}$ inside each of the following strips:
\begin{align*} 
&M_i \cap \left\{(x,y):x > o_1+d, 0 < y-o_2-d< \frac{d}{2^{i+2}}\right\}; \\
&M_i \cap \left\{(x,y):0 < x-o_1-d< \frac{d}{2^{i+2}}, y > o_2+d\right\}; 
\end{align*}

\item[\textcircled{2}] $m/2$ squares $Q_{(m+1)j}$ into the interval $[\mathbf{p}_1,\mathbf{p}_2]$ with endpoints $\mathbf{p}_1$,$\mathbf{p}_2$ on the line $y=o_2+d+(x-o_1-d)/2$ and so that $\rho(\mathbf{p}_1)=\frac{4}{3}\cdot\frac{d}{2^i}, \rho(\mathbf{p}_2)=\frac{5}{3}\cdot\frac{d}{2^i}$;

\item[\textcircled{3}] $m/2$ squares $Q_{(m+1)j}$ into the interval $[\mathbf{q}_1,\mathbf{q}_2]$ with endpoints $\mathbf{q}_1$,$\mathbf{q}_2$ on the line $y=o_2+d+2(x-o_1-d)$ and so that $\rho(\mathbf{q}_1)=\frac{4}{3}\cdot\frac{d}{2^i}, \rho(\mathbf{q}_2)=\frac{5}{3}\cdot\frac{d}{2^i}$.
\end{itemize}

$\bf{Case\ 2}$: if $m$ is odd, then we put

\begin{itemize}
\item[\textcircled{4}] $(m+1)/2$ squares $Q_{(m+1)j}$ inside each of the strips in \textcircled{1};

\item[\textcircled{5}] $(m+1)/2$ squares $Q_{(m+1)j}$ into the interval $[\mathbf{p}_1,\mathbf{p}_2]$ with endpoints $\mathbf{p}_1$,$\mathbf{p}_2$ on the line $y=o_2+d+(x-o_1-d)/2$ and so that $\rho(\mathbf{p}_1)=\frac{4}{3}\cdot\frac{d}{2^i}, \rho(\mathbf{p}_2)=\frac{5}{3}\cdot\frac{d}{2^i}$;

\item[\textcircled{6}] $(m+1)/2$ squares $Q_{(m+1)j}$ into the interval $[\mathbf{q}_1,\mathbf{q}_2]$ with endpoints $\mathbf{q}_1$,$\mathbf{q}_2$ on the line $y=o_2+d+2(x-o_1-d)$ and so that $\rho(\mathbf{q}_1)=\frac{4}{3}\cdot\frac{d}{2^i}, \rho(\mathbf{q}_2)=\frac{5}{3}\cdot\frac{d}{2^i}$.
\end{itemize}

On the square $A_1B_1C_1D_1\subset Q_{mk}$, we define the function $f_m$ by $f_m(\mathbf{x})=\tilde{f}_m(\rho(\mathbf{x}))$, where
\begin{equation*}
\tilde{f}_m(\rho)=
\begin{cases}
0, & {\text {if } } \rho \geq d;\\
\frac{i\om_m}m , & {\text {if } }  \rho = d/2^i, i=1,2,\dots,m-1;\\
\om_m ,& {\text {if } }  \rho = 0;\\
\end{cases}
\end{equation*}
Moreover, on the intervals $[a_i,b_i]$, $i=1,2,...,m$, we define $\tilde{f}_m$ by 
\begin{equation*}
\tilde{f}_m(\rho)=
\begin{cases}
\frac{(\tilde{f}_m(a_i)+\tilde{f}_m(b_i))}2,& {\text {if } } \rho \in [\frac{2a_i+b_i}{3},\frac{a_i+2b_i}{3}]; \\
\text{linear, }& {\text {if } } \rho \in [a_i,\frac{2a_i+b_i}{3}] \cup  [\frac{a_i+2b_i}{3},b_i].
\end{cases}
\end{equation*}
We perform the same construction in the square $A_3B_3C_3D_3$.

On the polygon $B_1C_1D_1B_3C_3D_3$ the function $f_m$ is constant along the lines with slope 1. Outside the polygon $B_1A_1D_1B_3A_3D_3$ the  function $f_m$ is  0.
We define the function $$f=\sum\limits_{m=1}^\infty f_m.$$

For every $x\in \Om$, there exists the $m_x\in \mathbb{N}$ such that for all $m\ge m_x$, $f_m(x)=0$. Thus $f$ is well defined.

\vspace{2mm}

{\bf Claim 1:} $f$ is differentiable almost everywhere. 

\vspace{2mm}

By  the Rademacher-Stepanov theorem (see e.g. \cite[Theorem 3.1.9]{Federer1969}, a short proof in \cite{Maly1999b}), it is enough to prove that $\lip(f,x)<\infty$ for a.e. $x\in \Om$, where $\lip(f,x)$ was defined in \eqref{lip}.

For every $x\in \Om$, there exists the smallest $m_x\in \mathbb{N}$ such that for all $m\ge m_x$, $f_m(x)=0$ and \buo  $m_x>1$. Moreover for every $x$  there exists a neighborhood $B(x,r)$ such that for any $y\in B(x,r)$ and any $m\ne m_x$, we have $f_m(x)=f_m(y)$. Thus $\lip(f,x)=\lip(f_{m_x},x)$. Since functions $f_m$ are Lipschitz for every $m$, we conclude that $\lip(f_{m_x},x)$, and thus also
$\lip(f,x)$, is finite. 

\vspace{2mm}

{\bf Claim 2:} $f \in W^{1,2}(\R^2)$ and $f$ satisfies the Luzin (N) property.

\vspace{2mm}

Since each function $f_m$ is Lipschitz,
$f_m\in W^{1,2}(\R^2)$,  and thus $f$ as a weak limit of finite sums of $f_m$'s also
belongs to $ W^{1,2}(\R^2)$.

Similarly, each $f_m$ satisfies the Luzin (N) property, and thus $f$ as a countable sum of $f_m$'s also satisfies the Luzin (N) property.


\vspace{2mm}

{\bf Claim 3:} $f \notin \frac12\po\ACq^{2}(\Om,\R)$.

\vspace{2mm}

By \eqref{12regular}  for each $m\in\N$, and $k=1,2,\dots,r_m$, the disjoint intervals $T_{mk} \subset Q_{mk}$ are $\frac12$-regular, and by \eqref{sizeQmk}, for every $\de>0$, there exists $m_0\in \N$ so that
$$\LL^2\left(\bigcup_{m=m_0}^\infty\bigcup_{k=1}^{r_m} T_{mk}\right)<\frac1{7\cdot8^{m_0-2}}\LL^2(Q_{11})<\de.$$

Since, for each $m\in\N$, and $k=1,2,\dots,r_m$,  $|f(T_{mk})|=\om_m $, we have
\begin{align*}
\sum\limits_{m=m_0}^{\infty}\sum\limits_{k=1}^{r_m}|f(T_{mk})|^2
=\sum\limits_{m=m_0}^{\infty}\sum\limits_{k=1}^{r_m} \om^2_m = \sum\limits_{m=m_0}^{\infty}r_m\om_m^2
= \sum\limits_{m=m_0}^{\infty} \frac{1}{4m}=\infty.
\end{align*}
Therefore, $f\notin \frac12\po\ACq^{2}(\Om,\R)$.

\vspace{2mm}

{\bf Claim 4:} $f \in 1\po\ACq^{2}(\Om,\R)\setminus \ACq_H^{2}(\Om,\R)$.

\vspace{2mm}

By  Theorem~\ref{1capH} and Claim 3, it is enough to show that $f$ is  in $1\po\ACq^{2}(\Om,\R)$.

To see this, first note that, by an adaptation of \cite[Lemma 3]{Csornyei2000}, for every 1-regular interval $I=[\mathbf{a},\mathbf{b}]$, there exists an index $m=m(I)$, so that, 
$$
|f(\mathbf{a})-f(\mathbf{b})|^2 \leq 16 |f_{m(I)}(\mathbf{a})-f_{m(I)}(\mathbf{b})|^2.
$$
Let
$$
\mathcal{D}_1=
\left\{ I=[\mathbf{a},\mathbf{b}]:I\text{ is }1\po\text{regular and }|f_{m(I)}(\mathbf{a})-f_{m(I)}(\mathbf{b})| \leq 9\frac{\om_{m(I)}}{m(I)}\right\},
$$
$$
\mathcal{D}_2=
\left\{I=[\mathbf{a},\mathbf{b}]:I\text{ is }1\po\text{regular and }|f_{m(I)}(\mathbf{a})-f_{m(I)}(\mathbf{b})| > 9\frac{\om_{m(I)}}{m(I)}\right\}.
$$

We will prove that there exist two measures $\mu_1$ and $\mu_2$, absolutely continuous with respect to the Lebesgue measure, such that 
\begin{equation}\label{muone}
|f_{m(I)}(\mathbf{a})-f_{m(I)}(\mathbf{b})|^2 \leq \mu_1([\mathbf{a},\mathbf{b}])
\end{equation}
for each $[\mathbf{a},\mathbf{b}] \in \mathcal{D}_1$, and 
\begin{equation}\label{mutwo}
|f_{m(I)}(\mathbf{a})-f_{m(I)}(\mathbf{b})|^2 \leq \mu_2([\mathbf{a},\mathbf{b}])
\end{equation}
for each $[\mathbf{a},\mathbf{b}] \in \mathcal{D}_2$.\\

If such measures exist, then the absolute continuity of $\mu_1$ and $\mu_2$ implies that for all $\e>0$, there exists $\de>0$, such that for each finite collection of non-overlapping 1-regular intervals $\{[\mathbf{a}_j,\mathbf{b}_j]\}$ with $\LL^2 (\bigcup\limits_{j}^{}[\mathbf{a}_j,\mathbf{b}_j])<\de$,
$$
\mu_1(\bigcup\limits_{j}^{}[\mathbf{a}_j,\mathbf{b}_j]) < \frac{\e}{32}\text{,       }\ \ \mu_2(\bigcup\limits_{j}^{}[\mathbf{a}_j,\mathbf{b}_j]) < \frac{\e}{32}.
$$
 Hence, we obtain 
\begin{align*}
\sum\limits_{j}^{} |f(\mathbf{a}_j)-f(\mathbf{b}_j)|^2 \leq 16\sum\limits_{j}^{} |f_{m([\mathbf{a}_j,\mathbf{b}_j])}(\mathbf{a}_j)-f_{m([\mathbf{a}_j,\mathbf{b}_j])}(\mathbf{b}_j)|^2 \\
 \leq 16\sum\limits_{j}^{}\mu_1([\mathbf{a}_j,\mathbf{b}_j])+16\sum\limits_{j}^{}\mu_2([\mathbf{a}_j,\mathbf{b}_j]) \\
 = 16\mu_1\Big(\bigcup\limits_{j}^{}[\mathbf{a}_j,\mathbf{b}_j]\Big) +16\mu_2\Big(\bigcup\limits_{j}^{}[\mathbf{a}_j,\mathbf{b}_j]\Big) < \e,
\end{align*}
which proves that $f \in 1\po\ACq^{2}(\Om,\R)$.

Thus, to complete the proof, it is enough to prove the existence of measures $\mu_1$ and $\mu_2$. 

\vspace{2mm}

{\bf{Existence of the measure $\mu_1$}}

\vspace{2mm}

For a fixed 1-regular interval $I=[\mathbf{a},\mathbf{b}]\in \mathcal{D}_1$, let $m=m(I)$ be such that $|f(\mathbf{a})-f(\mathbf{b})| \leq 4|f_{m(I)}(\mathbf{a})-f_{m(I)}(\mathbf{b})|$. 
If $|f_m(\mathbf{a})-f_m(\mathbf{b})|=0$, then we can remove this interval $[\mathbf{a},\mathbf{b}]$ without affecting our results. If $|f_m(\mathbf{a})-f_m(\mathbf{b})|>0$, 
then let $I'=[\mathbf{a}',\mathbf{b}']$ be the smallest 1-regular sub-interval of $I$ such that $|f_m(\mathbf{a}')-f_m(\mathbf{b}')| = |f_m(\mathbf{a})-f_m(\mathbf{b})|$. Then $I' \subset Q_{mk}$ for some $k$, and there are three cases:
\begin{itemize}
\item[\textcircled{a}] Both points $\mathbf{a}'$ and  $\mathbf{b}'$ are in $\bigtriangleup B_1C_1D_1$. 

\item[\textcircled{b}] Both points $\mathbf{a}'$ and   $\mathbf{b}'$ are in $\bigtriangleup B_3C_3D_3$. 

\item[\textcircled{c}] Point $\mathbf{b}'$ is in $\bigtriangleup B_3C_3D_3$, and $\mathbf{a}'$ is in $\bigtriangleup B_1C_1D_1$. 
\end{itemize}
In case \textcircled{c}, let $F_1$ be the intersection point of line $\mathbf{a}'\mathbf{b}'$ and the side $B_1C_1$ (or $C_1D_1$), and $I^*=[F_1,\mathbf{b}']\subset I'$. If  $|f_m(F_1)-f_m(\mathbf{b}')| > 9\frac{\om_m}{m}$, then we put this interval $I^*$ into set $\mathcal{D}_2$. If $|f_m(F_1)-f_m(\mathbf{b}')| \leq 9\frac{\om_m}{m}$, then we set $I'=I^*$.
Therefore, combining these cases, \buo, we can assume that $I'\subset \square A_1B_1C_1D_1$. 

For $1\leq i \leq m-9$, we set 
$$
S_i = \Big\{ \mathbf{x}\in \R^2: \rho(\mathbf{x}) \in \bigcup\limits_{j=i}^{i+9} [a_j,b_j]\Big\},
$$
Since $|f_m(\mathbf{a}')-f_m(\mathbf{b}')|\leq 9\frac{\om_m}{m}$, we have $\mathbf{a}',\mathbf{b}'\in S_i$ for some integer  $1\leq i \leq m-9$. Now notice that $f_m$ is Lipschitz on $\bigcup_{j=i}^{i+9} [a_j,b_j]$ with Lipschitz constant 
\begin{align*}
K&=\frac{\om_m/2m}{\frac{1}{3}\cdot\frac{\sqrt{2}}{2}\cdot(b_{i+9}-a_{i+9})} \leq \frac{\om_m/2m}{\frac{1}{3}\cdot\frac{\sqrt{2}}{2}\cdot(d/2^{i+9-1}-d/2^{i+9})} \\
&= 3\cdot 2^{9-1}\cdot \sqrt{2}\cdot \frac{\om_m}{m} \cdot \frac{1}{d/2^i}.
\end{align*}
Therefore
$$
|f_m(\mathbf{a}')-f_m(\mathbf{b}')|^2 \leq 9\cdot 2^{17} \cdot \frac{\om_m^2}{m^2} \cdot \frac{(\diam I')^2}{(d/2^i)^2}.
$$
Since 
$
(\diam I')^2 \leq 2 \LL^2(I')$  and $\LL^2(S_i)< 8\cdot (d/2^i)^2, 
$
we have 
\begin{align*}
|f_m(\mathbf{b}')-f_m(\mathbf{a}')|^2 &\leq 9\cdot 2^{21}\cdot\frac{\om_m^2}{m^2}\cdot\frac{\LL^2(I')}{\LL^2(S_i)}\\
&=\int\limits_{I'}^{}9\cdot2^{21}\cdot\frac{\om_m^2}{m^2}\cdot\frac{1}{\LL^2(S_i)}.
\end{align*}
Thus if we set $\mu_1=\int g$, where
$$
g(\mathbf{x})=9\cdot2^{21}\cdot\sum\limits_{m=1}^\infty \sum\limits_{k=1}^{r_m} \sum\limits_{i=1}^{m-9} \left(\frac{\om_m^2}{m^2}\cdot\frac{\chi_{s_i}(\mathbf{x})}{\LL^2(S_i)}\right)\in L^1(\R^2),
$$
then $\mu_1$ satisfies \eqref{muone} for each interval $[\mathbf{a},\mathbf{b}]\in\mathcal{D}_1$ (cf. \cite[p.~154]{Csornyei2000} for details).

\vspace{2mm}

{\bf Existence of the measure $\mu_2$}

\vspace{2mm}

Let $\mu_2$ be an absolutely continuous measure for which
$$
\mu_2(Q_{mk})=\frac{4}{m\cdot r_m},\ \  m\in \mathbb{N}, k=1,\dots,r_m.
$$
This measure exists because
$$
\sum\limits_{j=1}^{r_{m+1}/r_m} \mu_2(Q_{(m+1)j})=\frac{4}{(m+1)\cdot r_m} < \frac{4}{m\cdot r_m} = \mu_2(Q_{mk}),
$$
and
$$
\sum\limits_{j=1}^{r_m} \mu_2(Q_{mk})=\frac{4}{m} \rightarrow 0.
$$

As before, for a fixed 1-regular interval $I=[\mathbf{a},\mathbf{b}] \in \mathcal{D}_2$, let $m=m(I)$ be such that $|f(\mathbf{a})-f(\mathbf{b})| \leq 4|f_{m(I)}(\mathbf{a})-f_{m(I)}(\mathbf{b})|$. 
Let $I'=[\mathbf{a}',\mathbf{b}']$ be the smallest 1-regular  sub-interval of $I$ such that $|f_m(\mathbf{a}')-f_m(\mathbf{b}')| = |f_m(\mathbf{a})-f_m(\mathbf{b})|$. Then $I' \subset Q_{mk}$ for some $k$, and  by the argument above, \buoo, we can assume that $I'\subset \square A_1B_1C_1D_1$. 

Let 
$$\beta = \frac{|f(\mathbf{a})-f(\mathbf{b})|\cdot m}{\om_m}.$$ 
Since $I\in \mathcal{D}_2$, we have $\beta \geq 9$. 
For simplicity, we can assume that $C_1=(0,0)$. Let $j$ be the smallest integer with $\rho(\mathbf{b}') \geq \frac{d}{2^j}$, and  $i$ be the biggest integer with $\rho(\mathbf{a}')\geq \frac{d}{2^{j+i}}$.
We denote $\mathbf{a}'=(a_1,a_2),\mathbf{b}'=(b_1,b_2)$, 
and, for $1\leq v \leq i$, we set 
$$
S_v=\Big\{\mathbf{x}\in \R^n: \rho(\mathbf{x})\in \Big[\frac{d}{2^{j+v}},\frac{d}{2^{j+v-1}}\Big]\Big\}.
$$

Let $\mathbf{a}_{1t}$ and $\mathbf{a}_{2t}$ be the points on the lines $y=2x$ and $y=x/2$, respectively, so that 
$$
\rho (\mathbf{a}_{1t})=\rho (\mathbf{a}_{2t})= \frac{t}{3}\cdot\frac{d}{2^{j+i}}\text{, with }4\leq t \leq 5.
$$
Let $\mathbf{b}_{1t}$, $\mathbf{b}_{2t}$ and $\mathbf{c}_{1t}$, $\mathbf{c}_{2t}$ be the images of the orthogonal projections of $\mathbf{a}_{1t}$, $\mathbf{a}_{2t}$ onto the horizontal line and onto the vertical line through $C_1$, respectively.
Finally let $\mathbf{d}_t=(\frac{t}{3}\cdot\frac{d}{2^{j+v}},0)\text{, }\mathbf{f}_t=(0,\frac{t}{3}\cdot\frac{d}{2^{j+v}})$.
Since 
$$
|C_1-\mathbf{c}_{1t}|=|\mathbf{a}_{1t}-\mathbf{b}_{1t}|=2|C_1-\mathbf{b}_{1t}|=2|\mathbf{a}_{1t}-\mathbf{c}_{1t}|,
$$
$$
|\mathbf{b}_{1t}-\mathbf{d}_t|=|\mathbf{a}_{1t}-\mathbf{b}_{1t}|\text{, and } |\mathbf{a}_{1t}-\mathbf{c}_{1t}|=|\mathbf{f}_{t}-\mathbf{c}_{1t}|,
$$
we have
\begin{equation}\label{ut1}
\frac{t}{3}\cdot\frac{d}{2^{j+v}}=|C_1-\mathbf{d}_{t}|=3|C_1-\mathbf{b}_{1t}|,
\end{equation}
and
\begin{equation}
\label{ut2}
\frac{t}{3}\cdot\frac{d}{2^{j+v}}=|C_1-\mathbf{f}_{t}|=3|\mathbf{a}_{1t}-\mathbf{c}_{1t}|.
\end{equation}
Moreover, since
$$
|\mathbf{a}_{2t}-\mathbf{c}_{2t}|=|C_1-\mathbf{b}_{2t}|=2|\mathbf{a}_{2t}-\mathbf{b}_{2t}|=2|C_1-\mathbf{c}_{2t}|,
$$
$$
|\mathbf{a}_{2t}-\mathbf{b}_{2t}|=|\mathbf{b}_{2t}-\mathbf{d}_{t}|\text{, and } |\mathbf{c}_{2t}-\mathbf{f}_{t}|=|\mathbf{a}_{2t}-\mathbf{c}_{2t}|,
$$
we have
\begin{equation}
\label{ut3}
\frac{t}{3}\cdot\frac{d}{2^{j+v}}=|C_1-\mathbf{d}_{t}|=3|\mathbf{a}_{2t}-\mathbf{b}_{2t}|,
\end{equation}
and
\begin{equation}\label{ut4}
\frac{t}{3}\cdot\frac{d}{2^{j+v}}=|C_1-\mathbf{f}_{t}|=3|C_1-\mathbf{c}_{2t}|.
\end{equation}
Therefore, if $v \leq i-1$, by \eqref{ut1}, 
$$
|\mathbf{a}_{1t}-\mathbf{b}_{1t}|=2\cdot|C_1-\mathbf{b}_{1t}|=\frac{t}{3^2}\cdot\frac{d}{2^{j+v-1}} > \frac{4}{3^2}\cdot\frac{d}{2^{j+v-1}} > \frac{d}{2^{j+v-1}}\geq \frac{d}{2^{j+i}} > a_2,
$$
and, by \eqref{ut4}, 
$$
|\mathbf{a}_{2t}-\mathbf{c}_{2t}|=2\cdot|C_1-\mathbf{c}_{2t}|=\frac{t}{3^2}\cdot\frac{d}{2^{j+v-1}} > \frac{4}{3^2}\cdot\frac{d}{2^{j+v-1}} > \frac{d}{2^{j+v-1}}\geq \frac{d}{2^{j+i}} > a_1.
$$
Moreover, if $v \leq i-2$, by \eqref{ut3}  
$$
|\mathbf{a}_{2t}-\mathbf{b}_{2t}|=\frac{1}{3}\cdot|C_1-\mathbf{d}_{t}|=\frac{t}{3^2}\cdot\frac{d}{2^{j+v}} > \frac{4}{3^2}\cdot\frac{d}{2^{j+v}} > \frac{d}{2^{j+v+2}}\geq \frac{d}{2^{j+i}} > a_2,
$$
and, by  \eqref{ut2}, 
$$
|\mathbf{a}_{1t}-\mathbf{c}_{1t}|=\frac{1}{3}\cdot|C_1-\mathbf{f}_{t}|=\frac{t}{3^2}\cdot\frac{d}{2^{j+v}} > \frac{4}{3^2}\cdot\frac{d}{2^{j+v}} > \frac{d}{2^{j+v}}\geq \frac{d}{2^{j+i}} > a_1.
$$

Thus, to see that the points $\mathbf{a}_{1t},\mathbf{a}_{2t}$ belong to the interval $I'$, it suffices to show that 
$$
|\mathbf{a}_{1t}-\mathbf{b}_{1t}|<b_2\text{ and }|\mathbf{a}_{2t}-\mathbf{b}_{2t}|<b_2;
$$
$$
|\mathbf{a}_{1t}-\mathbf{c}_{1t}|<b_1\text{ and }|\mathbf{a}_{2t}-\mathbf{c}_{2t}|<b_1.
$$
First of all, we have
$$
b_1>\frac{d}{2^{j+3}} \text{ and } b_2>\frac{d}{2^{j+3}}.
$$

Indeed, if $b_1\leq \frac{d}{2^{j+3}}$, then $b_2>\frac{d}{2^{j+1}}$. Thus
$$
b_2-a_2>\frac{d}{2^{j+1}}-\frac{d}{2^{j+i}}>\frac{d}{2^{j+3}}\geq b_1 > b_1-a_1,
$$
which  contradicts  1-regularity of $[\mathbf{a}',\mathbf{b}']$.
Similarly  $b_2\leq \frac{d}{2^{j+3}}$.

Therefore, if $v\geq 4$, we have 
\begin{align*}
&|\mathbf{a}_{1t}-\mathbf{b}_{1t}|<\frac{5}{3^2}\cdot\frac{d}{2^{j+v-1}} < \frac{d}{2^{j+v-1}}\leq \frac{d}{2^{j+3}} < b_2;\\
&|\mathbf{a}_{2t}-\mathbf{b}_{2t}|<\frac{5}{3^2}\cdot\frac{d}{2^{j+v}} < \frac{d}{2^{j+v}}\leq \frac{d}{2^{j+3}} < b_2;\\
&|\mathbf{a}_{1t}-\mathbf{c}_{1t}|<\frac{5}{3^2}\cdot\frac{d}{2^{j+v}} < \frac{d}{2^{j+v}}\leq \frac{d}{2^{j+3}} < b_1;\\
&|\mathbf{a}_{1t}-\mathbf{c}_{1t}|<\frac{5}{3^2}\cdot\frac{d}{2^{j+v-1}} < \frac{d}{2^{j+v-1}}\leq \frac{d}{2^{j+3}} < b_1.
\end{align*}

In conclusion, we have proved that, if $4\leq v \leq i-2$, then the interval $I'$ contains the points $\mathbf{a}_{1t},\mathbf{a}_{2t}$ for all $4\leq t \leq 5$. Thus $I'$ contains the intervals $[\mathbf{a}_{14},\mathbf{a}_{15}] $ and $[\mathbf{a}_{24},\mathbf{a}_{25}] $. Therefore, $I'$ covers at least $(i-5)m/2$ of the $r_{m+1}$ squares $Q_{(m+1)h}$, $1\leq h \leq r_{m+1}$. Thus
$$
\mu_2([\mathbf{a}',\mathbf{b}'])\geq (i-5)m\cdot\frac{4}{(m+1)r_{m+1}}=(i-5)\cdot\frac{4m}{(m+1)^2}\cdot\om_m^2,
$$
and, since $\beta\leq m$ and $\beta\leq i+2$, we have
\begin{align*} 
|f(\mathbf{a}')-f(\mathbf{b}')|^2 =\frac{\beta^2}{m^2}\cdot\om_m^2\leq \frac{\beta^2}{m^2}\cdot\frac{(m+1)^2}{4m(i-5)}\cdot\mu_2([\mathbf{a}',\mathbf{b}'])\\
< \frac{(i+2)(m+1)^2}{4(i-5)m^2}\cdot\mu_2([\mathbf{a}',\mathbf{b}']).
\end{align*}
Moreover, since $m\geq 9$, we have $5m^2-22m+7>0$. Since $i\geq 9$, we have $i(3m^2-2m+1)>22m^2+4m+2$. Thus, $(i+2)(m+1)^2<2(i-5)m^2$, and hence 
$$
|f(\mathbf{a}')-f(\mathbf{b}')|^2<\mu_2([\mathbf{a}',\mathbf{b}']).
$$
\end{proof}

\begin{ack}
We thank Jan Mal\'y for bringing this problem to our attention and  Olga Maleva for helpful discussions.
\end{ack}


\def\cprime{$'$}

\end{document}